\documentclass[12pt,oneside,reqno]{amsart}

\usepackage{tikz}
\usetikzlibrary{arrows.meta}
\usetikzlibrary{bending}

\usepackage{hyperref}
\usepackage{amsxtra}
\usepackage{amsopn}
\usepackage{amsmath,amsthm,amssymb}
\usepackage{color}
\usepackage{amscd}
\usepackage{pifont}
\usepackage{amsfonts}
\usepackage{latexsym}
\usepackage{verbatim}
\usepackage{pb-diagram}
\usepackage{tikz}

\newcommand{\fr}{\mathfrak}

\newcommand{\op}{\operatorname}

 \newtheorem{lemma} {Lemma} [section]
\newtheorem{theorem}[lemma]{Theorem} 
\newtheorem{remark}[lemma] {Remark} 
\newtheorem{prop} [lemma]{Proposition}  
\newtheorem{definition}[lemma] {Definition} 
\newtheorem{corol}[lemma] {Corollary} 
\newtheorem{example}[lemma] {Example} 

\newtheorem{question}[lemma] {Question}

\begin{document}
\title{Einstein Lie groups, geodesic orbit manifolds and regular Lie subgroups} 

\author{Nikolaos Panagiotis Souris}
\address{University of Patras, Department of Mathematics, University Campus, 26504, Rio Patras, Greece}
\email{nsouris@upatras.gr}

\begin{abstract}
We study the relation between two special classes of Riemannian Lie groups $G$ with a left-invariant metric $g$: The Einstein Lie groups, defined by the condition $\op{Ric}_g=cg$, and the geodesic orbit Lie groups, defined by the property that any geodesic is the integral curve of a Killing vector field. The main results imply that extensive classes of compact simple Einstein Lie groups $(G,g)$ are not geodesic orbit manifolds, thus providing large-scale answers to a relevant question of Y. Nikonorov.  Our approach involves studying and characterizing the $G\times K$-invariant geodesic orbit metrics on Lie groups $G$ for a wide class of subgroups $K$ that we call (weakly) regular. By-products of our work are structural and characterization results that are of independent interest for the classification problem of geodesic orbit manifolds.



\end{abstract}
\maketitle

\section{Introduction}
Two notable and widely studied classes of Riemannian manifolds $(M,g)$ are the \emph{Einstein manifolds} (\cite{Bes}) and the \emph{geodesic orbit manifolds}, or \emph{g.o. manifolds} (\cite{BerNik0}).  The former are defined by the condition $\op{Ric}_g=cg$, $c\in \mathbb R$, where $\op{Ric}_g$ is the Ricci tensor. The latter are defined by the property that any geodesic is the integral curve of a complete Killing vector field (Definition \ref{defg.o.}). One of the largest and most important proper subclasses of g.o. manifolds are the \emph{naturally reductive manifolds} (Definition \ref{defnat}), in which case the metric $g$ is essentially induced from a bi-invariant (pseudo-)Riemannian metric of a Lie group $G$ acting transitively and isometrically on the manifold (\cite{DaZi}, \cite{Kos}).      

The classes of Einstein manifolds and g.o. manifolds have an extensive intersection. For example, irreducible \emph{Riemannian symmetric spaces} (and more generally, \emph{isotropy irreducible spaces} \cite{WaZi}) are both Einstein and g.o manifolds. Furthermore, D'Atri and Ziller discovered large families of compact Einstein Lie groups $(G,g)$, where $g$ is a left-invariant metric, that are also naturally reductive  (\cite{DaZi}) and hence geodesic orbit. On the other hand, several aspects of the relation between Einstein Lie groups and g.o. manifolds remain unknown.  A relevant well-known question, stated by D'Atri and Ziller in 1979, can be reformulated as follows.

\begin{question}\label{problem1} \emph{(\cite{DaZi})} Which compact Einstein Lie groups $(G,g)$ are not naturally reductive? \end{question}

\noindent Question \ref{problem1} has attracted substantial interest, resulting in the discovery of a plethora of left-invariant, non-naturally reductive Einstein metrics on compact simple Lie groups, mainly in the works \cite{ArMo}, \cite{ArSaSt}, \cite{ArSaSt1}, \cite{CheChe}, \cite{CheCheArx}, \cite{CheCheDe}, \cite{CheCheDe3}, \cite{CheLi}, \cite{ChSa}, \cite{Mori}, \cite{YanDen}, \cite{ZaBo}, \cite{ZaCheDe}, \cite{ZaCheDe1}. Recently, Nikonorov stated the following natural generalization of Question \ref{problem1}.

\begin{question}\label{problem2}\emph{(\cite{NikEin})} Which compact Einstein Lie groups $(G,g)$ are not geodesic orbit? \end{question}

\noindent Nikonorov also provided the first answer to Question \ref{problem2}, by obtaining a left-invariant Einstein metric on the compact Lie group $G_2$ that is not geodesic orbit.  Further answers were given in the works \cite{CheCheDe2} and \cite{XuTa}. In the former work it is proved that any of the non-naturally reductive Einstein Lie groups in \cite{CheCheDe}, \cite{CheLi} and \cite{ZaCheDe} is not geodesic orbit.  Question \ref{problem2} is generally more complicated than Question \ref{problem1}, while the approach used so far mainly involves case-by-case examination for suitably chosen non-naturally reductive Einstein metrics in literature.
  
In this paper, we study Question \ref{problem2} from a new perspective by looking at the local structure of the isometry groups of the Riemannian Lie groups $(G,g)$. As follows from the Ochiai-Takahashi Theorem, the connected isometry group of a compact simple Lie group $(G,g)$ is locally isomorphic to a product $G\times K$, where $K$ is a closed connected subgroup of $G$ (\cite{Ot}, \cite{DaZi}).  More generally, a Riemannian metric $g$ on $G$ is called \emph{$G\times K$-invariant} if it is invariant by the left translations in $G$ and the right translations in $K$. At the heart of our study lies a class of Lie subgroups $K$ of $G$ that we call regular, based on the definition of a regular subalgebra in the sense of Dynkin (\cite{Dyn}, \cite{DieGr}).

\begin{definition}\label{defregular} A Lie subgroup $K$ of $G$ is called regular if the Lie algebra $\fr{k}$ of $K$ is normalized by a Cartan subalgebra of the Lie algebra $\fr{g}$ of $G$.  Equivalently, $K$ is called regular if there exists a Cartan subalgebra $\fr{t}$ of $\fr{g}$ such that $[\fr{k},\fr{t}]\subseteq \fr{k}$.
\end{definition}

Regular subgroups comprise an extensive class that includes the subgroups of maximal rank and the abelian subgroups. Essentially, regular subgroups are (up to local isomorphism) the normal subgroups of all subgroups of $G$ that have maximal rank. Before we state our main result, a Riemannian Lie group $(G,g)$ will be called $G\times K$-geodesic orbit if any geodesic in $G$ is the orbit of an one parameter subgroup of $G\times K$ (see also Definition \ref{defg.o.}).

\begin{theorem}\label{consequence}Let $(G,g)$ be a connected compact Riemannian simple Lie group and assume that the metric $g$ is $G\times K$-invariant, where $K$ is a closed regular subgroup of $G$.  If $(G,g)$ is a $G\times K$-geodesic orbit manifold then $(G,g)$ is a naturally reductive manifold.  \end{theorem}   

\begin{corol}\label{consequence1}Let $(G,g)$ be a connected compact Riemannian simple Lie group and assume that the metric $g$ is $G\times K$-invariant, where $K$ is subgroup of $G$ that has maximal rank. Then $(G,g)$ is a g.o. manifold if and only if it is a naturally reductive manifold.\end{corol}

Theorem \ref{consequence} reduces the study of Question \ref{problem2} to the much simpler Question \ref{problem1} for extensive classes of Riemannian Lie groups.  This allows us to provide the following large-scale answer to Question \ref{problem2}.

\begin{theorem}\label{theoremlit1} Let $(G,g)$ be any (non-naturally reductive) Einstein Lie group obtained in \cite{ArMo}, \cite{ArSaSt1}, \cite{CheChe}, \cite{CheCheArx}, \cite{ChSa}, \cite{Mori} and \cite{ZaCheDe1}, or any non-naturally reductive Einstein Lie group induced from one of the standard triples $(\fr{sp}(2n_1n_2), \fr{sp}(n_1n_2)\times \fr{sp}(n_1n_2),2n_2\fr{sp}(n_1))$, $(\fr{f}_4,\fr{so}(9),\fr{so}(8))$, $(\fr{e}_8,\fr{so}(16),8\fr{su}(2))$, $(\fr{e}_8,\fr{so}(16),2\fr{so}(8))$ in \cite{YanDen}. Then $(G,g)$ is not a geodesic orbit manifold.    \end{theorem}

Theorem \ref{theoremlit1} can be extended to include additional Einstein Lie groups, by generalizing the notion of a regular subgroup to the more technical notion of a \emph{weakly regular subgroup} (Section \ref{weaklyregular}).  In particular, using the results of Section \ref{weaklyregular}, we obtain the following.

\begin{theorem}\label{extend}Let $(G,g)$ be any (non-naturally reductive) Einstein Lie group obtained in \cite{ArSaSt}, \cite{CheCheDe3}, \cite{ZaBo} or any non-naturally reductive Einstein Lie group induced from an irreducible triple $(\fr{g},\fr{k}, \fr{h})$ in \cite{YanDen}, with the possible exception of the triple $(\fr{so}(8),\fr{so}(7),\fr{g}_2)$.  Then $(G,g)$ is not a geodesic orbit manifold.\end{theorem}

Theorems \ref{theoremlit1} and \ref{extend} imply that the vast majority of non-naturally reductive Einstein Lie groups in the literature are not geodesic orbit manifolds.

\begin{remark}\label{remarkackn1} There exist non-naturally reductive Einstein metrics that are geodesic orbit.  For example, the flag manifolds $SO(2l)/U(l)$ and $Sp(l)/(U(1)\times Sp(l-1))$ respectively admit $SO(2l)$, $Sp(l)$-invariant Einstein metrics (see for example \cite{ArChr}).  Those metrics are geodesic orbit but in no way naturally reductive (\cite{AlAr}).\end{remark}

Apart from its application in Einstein Lie groups, Theorem \ref{consequence} is of independent interest for the problem of characterization and classification of g.o. manifolds, which has recently attracted significant interest (see \cite{BerNik0} and references therein). For example, the following corollary of Theorem \ref{consequence} characterizes the left-invariant geodesic orbit metrics on $SU(2)$.

\begin{corol}\label{Emil}Let $g$ be a left-invariant metric on a three dimensional compact simple Lie group $G$ (e.g. $G=SU(2)$).  Then $(G,g)$ is a geodesic orbit manifold if and only if it is a naturally reductive manifold.\end{corol} 

We note that Corollary \ref{Emil} is well - known.  Indeed, O. Kowalski and L. Vanhecke showed that all geodesic orbit Riemannian spaces $(M,g)$ with $\dim M\leq 5$ are naturally reductive (\cite{KoVa}).  More specifically, since $SU(2)$ is diffeomorphic to the sphere $S^3$, the conclusion of Lemma \ref{Emil} for $(SU(2),g)$ can also be verified from the classification of geodesic orbit Riemannian metrics on spheres in \cite{NikSphe}.   

\subsection{Overview of the proof and additional characterization results}

 To prove Theorem \ref{consequence}, we firstly show that if $K$ is a connected regular subgroup of a compact Lie group $G$ then any $G\times K$-geodesic orbit metric on $G$ splits into the sum of a bi-invariant metric on the Lie group $N^0(K)$ and a $G$-geodesic orbit metric on the homogeneous space $G/N^0(K)$, where $N^0(K)$ denotes the identity component of the normalizer of $K$ in $G$.

\begin{theorem}\label{split}Let $G$ be a compact Lie group with Lie algebra $\fr{g}$ and let $K$ be a connected regular subgroup of $G$ with Lie algebra $\fr{k}$.  Consider an orthogonal decomposition $\fr{g}=\fr{n}_{\fr{g}}(\fr{k})\oplus \fr{p}$ with respect to an $\op{Ad}$-invariant inner product on $\fr{g}$, where $\fr{n}_{\fr{g}}(\fr{k})$ is the normalizer of $\fr{k}$ in $\fr{g}$.  Then any $G\times K$-g.o. metric $g$ on $G$ splits as 

\begin{equation*}g=\left.g(\cdot \ , \cdot)\right|_{\fr{n}_{\fr{g}}(\fr{k})\times \fr{n}_{\fr{g}}(\fr{k})}+\left.g(\cdot \ , \cdot)\right|_{\fr{p}\times \fr{p}},\end{equation*}

\noindent where the restriction $\left.g(\cdot \ , \cdot)\right|_{\fr{n}_{\fr{g}}(\fr{k})\times \fr{n}_{\fr{g}}(\fr{k})}$ defines a bi-invariant metric on the Lie group $N^0(K)$ and the restriction $\left.g(\cdot \ , \cdot)\right|_{\fr{p}\times \fr{p}}$ defines a $G$-g.o. metric on the homogeneous space $G/N^0(K)$. \end{theorem}

Although regular subgroups comprise a wide class that admits diverse embeddings in $G$, they exhibit a consistent representation-theoretic property (Lemma \ref{key}), a result of which is Theorem \ref{split}.  It is that specific property that we use to extend the definition of a regular subgroup to that of a \emph{weakly regular subgroup} (Definition \ref{weaklyregulardef}).  \\

It should be mentioned that the property to be geodesic orbit implies additional symmetries. For example, we have the following lemma that we restate and prove in Section \ref{sec2} (Lemma \ref{NormalizerLemmaNew}). 

\begin{lemma}
Let $G$ be a compact Lie group with Lie algebra $\fr{g}$ and let $g$ be a $G\times K$-geodesic orbit metric on $G$. Denote by $\fr{k}$ the Lie algebra of $K$.  Then $g$ is not only $\op{ad}_{\fr{k}}$-equivariant, but also $\op{ad}_{\fr{n}_{\fr{g}}(\fr{k})}$-equivariant, where $\fr{n}_{\fr{g}}(\fr{k})$ denotes the normalizer of $\fr{k}$ in $\fr{g}$.\end{lemma}

A corollary of the above result is the well-known fact that any $G$-g.o. metric on a compact Lie group $G$ is necessarily bi-invariant (\cite{AlNik}).  Another consequence of the above lemma is the following.

 \begin{corol}\label{yields}Let $(G,g)$ be a compact simple Riemannian Lie group where $g$ is a left-invariant geodesic orbit metric. Let $\fr{g}\times \fr{k}$ be the Lie algebra of the isometry group of $(G,g)$.  Then $\fr{k}$ is a self-normalizing subalgebra of $\fr{g}$.\end{corol}

Finally, we mention a characterization result that we prove in Section \ref{weaklyregular}, using the notion of a weakly regular subgroup. 

\begin{prop}\label{isotropcorol} Let $(G,g)$ be a connected compact Riemannian simple Lie group and assume that the metric $g$ is $G\times K$-invariant, where $K$ is a subgroup of $G$ such that $G/K$ is strongly isotropy irreducible (i.e. the identity component of $K$ acts irreducibly on the tangent space of $G/K$ at the origin).   Then $(G,g)$ is a g.o. manifold if and only if it is a naturally reductive manifold.  \end{prop}

\subsection{Structure of the paper}
In Section \ref{sec1}, we present some preliminaries on Lie groups, homogeneous spaces and invariant metrics. In Section \ref{sec2}, we recall the notions of a geodesic orbit manifold and of a naturally reductive manifold, as well as some fundamental properties of geodesic orbit manifolds.  We also establish properties of $G\times K$-geodesic orbit metrics on compact Lie groups.  In Section \ref{sec3}, we discuss properties of regular subgroups. We prove a key representation-theoretic property (Lemma \ref{key}). We also prove Theorem \ref{split} and a useful corollary (Corollary \ref{ulcorol}).  Using the aforementioned corollary, in Section \ref{sec4.5} we prove the main result, Theorem \ref{consequence}.  In Section \ref{sec8}, we apply the main theorem to prove Corollary \ref{consequence1}, Theorem \ref{theoremlit1} and Corollary \ref{Emil}. Finally, in Section \ref{weaklyregular} we introduce the notion of a weakly regular subgroup and we prove Theorem \ref{extend}.\\

\noindent 
{\bf Acknowledgement.} The author had useful discussions with Sigmundur Gudmundsson, Ramiro Lafuente, Emilio Lauret and Jorge Lauret during the \emph{virtual seminar on geometry with symmetries}. Those discussions contributed to the overall improvement of the exposition, including Remark \ref{remarkackn1} and Corollary \ref{Emil}.  The author would also like to thank the anonymous referee for the improvement of the manuscript, including a simpler proof of Lemma 3.15.

\section{Preliminaries on invariant metrics on Lie groups and homogeneous spaces}\label{sec1}

\subsection{Left-invariant and $G\times K$-invariant metrics}  Let $G$ be a Lie group with Lie algebra $\fr{g}$. Denote by $\op{Ad}:G\rightarrow \op{Aut}(\fr{g})$ and $\op{ad}:\fr{g}\rightarrow \op{End}(\fr{g})$ the adjoint representations of $G$ and $\fr{g}$ respectively, where $\op{ad}_XY=[X,Y]$, $X,Y\in \fr{g}$.  We assume that $G$ is compact.  In this case, there exists an $\op{Ad}$-invariant inner product $Q$ on $\fr{g}$, which we fix. If $G$ is simple, the only $\op{Ad}$-invariant inner product (up to scaling) is the negative of the Killing form of $\fr{g}$.

  A Riemannian metric $g$ on $G$ is called \emph{left-invariant} if the left translations in $G$ are isometries of the manifold $(G,g)$.  A left-invariant metric is called \emph{bi-invariant} if the right translations in $G$ are also isometries of $(G,g)$.  The left-invariant Riemannian metrics $g$ are in bijection with inner products $\langle \ ,\ \rangle$ on $\fr{g}$ which in turn are in bijection with \emph{metric endomorphisms} $\Lambda=\Lambda_Q\in \op{End}(\fr{g})$ satisfying

\begin{equation}\label{metend}\langle X,Y\rangle = Q(\Lambda X,Y), \ \ X,Y\in \fr{g}.\end{equation}

\noindent  The Lie algebra $\fr{g}$ admits a $Q$-orthogonal decomposition $\fr{g}=\bigoplus_{j=1}^s\fr{g}_{\lambda_j}$, where $\fr{g}_{\lambda_j}$ are the distinct eigenspaces of $\Lambda$ with corresponding eigenvalues $\lambda_j>0$.  Without any loss of generality, we will make no distinction between a left-invariant metric $g$ and its corresponding metric endomorphism $\Lambda\in \op{End}(\fr{g})$.

Let $K$ be a closed subgroup of $G$ with Lie algebra $\fr{k}$.  The Lie group $G\times K$ acts transitively on $G$ by $(x,y)\cdot z:=xzy^{-1}$ where $x,z\in G$, $y\in K$.  A Riemannian metric $g$ on $G$ is called \emph{$G\times K$-invariant} if it is invariant by the aforementioned action of $G\times K$ on $G$. Equivalently, $g$ is a left-invariant metric that is also invariant by the right translations in $K$. Under this notation, bi-invariant metrics correspond to $G\times G$-invariant metrics.  The corresponding inner product $\langle \ ,\ \rangle$ of a $G\times K$-invariant metric is $\op{Ad}_K$-invariant (or equivalently any endomorphism $\op{ad}_{X}$, $X\in \fr{k}$, is skew-symmetric with respect to $\langle \ ,\ \rangle$ if $K$ is connected).    More generally, taking into account Equation \eqref{metend} and applying simple arguments from linear algebra, we obtain the following criterion.

\begin{lemma}\label{eigenspaceinvariant}Let $g$ be a left-invariant metric on $G$ with corresponding inner product $\langle \ ,\ \rangle$ and let $K$ be a closed subgroup of $G$.  The following are equivalent.\\
\noindent \emph{(i)} The metric $g$ is $G\times K$-invariant.\\
\noindent \emph{(ii)} The metric endomorphism $\Lambda$ of $g$ is $\op{Ad}_K$-equivariant, i.e. $\op{Ad}_k\circ \Lambda=\Lambda\circ \op{Ad}_k$ for all $k\in K$.\\
\noindent \emph{(iii)} Every eigenspace $\fr{g}_{\lambda}$ of $\Lambda$ in $\fr{g}$ is $\op{Ad}_K$-invariant.\end{lemma}

\noindent The following is the Lie-algebraic version of Lemma \ref{eigenspaceinvariant}.

\begin{lemma}\label{eigenspaceinvariant1}Let $g$ be a $G\times K$-invariant metric on $G$ with corresponding inner product $\langle \ ,\ \rangle$ and let $\fr{k}$ be the Lie algebra of $K$.  Then the following equivalent relations are true.\\
\noindent \emph{(iv}) Any endomorphism $\op{ad}_{X}$, $X\in \fr{k}$, is skew-symmetric with respect to $\langle \ ,\ \rangle$.\\
\noindent \emph{(v)} The metric endomorphism $\Lambda$ of $g$ is $\op{ad}_{\fr{k}}$-equivariant, i.e. $\op{ad}_X\circ \Lambda=\Lambda\circ \op{ad}_X$ for all $X\in \fr{k}$.\\
\noindent \emph{(vi)} Every eigenspace $\fr{g}_{\lambda}$ of $\Lambda$ in $\fr{g}$ is $\op{ad}_{\fr{k}}$-invariant.\end{lemma}

We remark that all statements (i) - (vi) of the above lemmas are equivalent if $K$ is connected.

\subsection{The general form of a $G\times K$-invariant metric on $G$}  We recall that for a (group or algebra) representation $\chi:K\rightarrow \op{End}(V)$, two $\chi$-submodules $W_1$ and $W_2$ of $V$ are called \emph{$\chi$-equivalent} if there exists a $\chi$-equivariant isomorphism $\phi:W_1\rightarrow W_2$.  When it becomes clear from the context, we will use the terminology \emph{equivalent} instead of \emph{$\chi$-equivalent}.

The general form of a metric endomorphism $\Lambda\in \op{End}(\fr{g})$, corresponding to a $G\times K$-invariant metric on the Lie group $G$, depends on the representation $\left.\op{Ad}\right|_K:K\rightarrow \op{Aut}(\fr{g})$ defined by  $\left.\op{Ad}\right|_K(k)X:=\op{Ad}_kX$ for $k\in K$ and $X\in \fr{g}$.  Since $K$ is compact, the Lie algebra $\fr{g}$ admits a $Q$-orthogonal decomposition into irreducible $\op{Ad}_K$-submodules. The pairwise equivalent submodules comprise the \emph{isotypic components} of $\left.\op{Ad}\right|_K$. More specifically, a subspace $V_i$ of $\fr{g}$ is called an isotypic component of $\left.\op{Ad}\right|_K$ if the following two conditions hold:\\

 \noindent (i) $V_i=\fr{g}_1\oplus \cdots \oplus \fr{g}_l$ where $\fr{g}_j$, $j=1,\dots ,l$, are pairwise equivalent, irreducible $\op{Ad}_K$-submodules.\\
\noindent  (ii) If $\fr{n}\subseteq \fr{g}$ is a submodule of $\left.\op{Ad}\right|_K$ that is equivalent to $\fr{g}_j$, $j=1,\dots,l$, then $\fr{n}\subseteq V_i$.\\
 
\noindent   The Lie algebra $\fr{g}$ admits a unique (canonical) $Q$-orthogonal decomposition

\begin{equation}\label{isotyp}\fr{g}=V_1\oplus \cdots \oplus V_N,\end{equation}

\noindent called the \emph{isotypic decomposition of $\left.\op{Ad}\right|_K$}, where $V_1,\dots, V_N$ are the isotypic components of $\left.\op{Ad}\right|_K$ (more generally, the isotypic decomposition exists and is canonical for any completely reducible representation, see e.g. \cite{Tel}, \cite{Sep}).  If $K$ is connected, the isotypic decomposition of $\left.\op{Ad}\right|_K$ coincides with the isotypic decomposition of the representation $\left.\op{ad}\right|_{\fr{k}}:\fr{k}\rightarrow \op{End}(\fr{g})$, defined by $\left.\op{ad}\right|_{\fr{k}}(X)Y=[X,Y]$, $X\in \fr{k}$ , $Y\in \fr{g}$. From Schur's lemma, it follows that any endomorphism $\Lambda$ of $\fr{g}$, corresponding to a $G\times K$-invariant metric on $G$ (in fact any $\op{Ad}_K$-equivariant endomorphism of $\fr{g}$), admits the block-diagonal form 

\begin{equation}\label{blockdiaginitial}\Lambda=\begin{pmatrix} 
 \left.\Lambda\right|_{V_1} & 0 & \cdots &0\\
 0& \left.\Lambda\right|_{V_2} &\cdots &0\\
  \vdots & \cdots & \ddots &\vdots\\
  0&\cdots &\cdots &\left.\Lambda\right|_{V_N}
  \end{pmatrix},
\end{equation}
   
\noindent with respect to the isotypic decomposition \eqref{isotyp}.  In particular, $\Lambda V_i\subseteq V_i$.  Moreover, if an isotypic component $V_i$ is $\op{Ad}_K$-irreducible then $\left.\Lambda\right|_{V_i}=\lambda_i\op{Id}$.  The same conclusion is true if $K$ is connected and the isotypic component $V_i$ is $\op{ad}_{\fr{k}}$-irreducible.  More generally, in studying invariant metrics we use the following result from linear algebra.

\begin{lemma}\label{equiv}
Let $\chi:H\rightarrow \op{End}(V)$ be a completely reducible Lie group (or Lie algebra) representation and let $Q$ be a $\chi$-invariant inner product on $V$.  Assume that $V$ admits the $Q$-orthogonal decomposition $V=W_1\oplus W_2$ where $W_1$ and $W_2$ are $\chi$-submodules such that any non-zero $\chi$-submodule of $W_1$ is inequivalent to any non-zero $\chi$-submodule of $W_2$.  Then any $\chi$-equivariant endomorphism $\Lambda:V\rightarrow V$ leaves the spaces $W_1$ and $W_2$ invariant.
\end{lemma}

\subsection{Invariant metrics on homogeneous spaces} Let $G$ be a compact Lie group and let $Q$ be an $\op{Ad}$-invariant inner product on its Lie algebra $\fr{g}$. Let $K$ be a closed subgroup of $G$ with Lie algebra $\fr{k}$ and consider the homogeneous space $G/K$. We consider the $Q$-orthogonal decomposition 

\begin{equation}\label{reddec1}\fr{g}=\fr{k}\oplus \fr{m}.\end{equation}

\noindent Then $\op{Ad}_{K}\fr{m}\subseteq \fr{m}$, and the space $\fr{m}$ is naturally identified with the tangent space $T_o(G/K)$ of $G/K$ at the origin $o:=eK$. A decomposition \eqref{reddec1} such that $\op{Ad}_K\fr{m}\subseteq \fr{m}$ is called a \emph{reductive decomposition} of $G/K$ with respect to the product $Q$ (in the general case, $Q$ is not required to be $\op{Ad}$-invariant).  A Riemannian metric $g$ on $G/K$ is called \emph{$G$-invariant} if it is invariant by the action of $G$ on $G/K$.  Equivalently, for any $x\in G$, the map $\tau_x:G/K\rightarrow G/K$ with $yK\mapsto (xy)K$ is an isometry of $(G/K,g)$.  The space $G/K$ equipped with a $G$-invariant Riemannian metric $g$ is called a \emph{Riemannian homogeneous space} $(G/K,g)$. The $G$-invariant Riemannian metrics $g$ on $G/K$ are in bijection with $\op{Ad}_K$-invariant inner products $\langle \ ,\ \rangle$ on $\fr{m}$, which in turn are in bijection with \emph{metric endomorphisms} $\Lambda=\Lambda_Q\in \op{End}(\fr{m})$ satisfying

\begin{equation}\label{metend1}\langle X,Y\rangle = Q(\Lambda X,Y), \ \ X,Y\in \fr{m}.\end{equation}

 The general form of the metric endomorphism $\Lambda$ depends on the isotypic decomposition $\fr{m}=\bigoplus_{i=1}^NV_i$ of the \emph{isotropy representation} $\left.\op{Ad}\right|_{K}:K\rightarrow \op{Gl}(\fr{m})$, given by $\left.\op{Ad}\right|_{K}(k)X=\op{Ad}_kX$, $k\in K$, $X\in \fr{m}$.  In particular, the $\op{Ad}_K$-equivariant endomorphism $\Lambda$ admits a similar form to \eqref{blockdiaginitial} with respect to the isotypic decomposition.      
 
 The following trivial type of metrics is of particular interest. 

\begin{definition}\label{defnorm}
Let $G/K$ be a homogeneous space.  A $G$-invariant metric $g$ on $G/K$ is called $G$-normal if there exists an $\op{Ad}$-invariant inner product $Q$ on $\fr{g}$ such that the corresponding metric endomorphism $\Lambda_Q$ of $g$ satisfies $\Lambda_Q=\lambda\op{Id}$, $\lambda>0$.  If $Q$ is the negative of the Killing form of $\fr{g}$ (up to scalar multiplication), the metric $g$ is called standard. 
\end{definition}

\section{Geodesic orbit manifolds}\label{sec2}

\subsection{Geodesic orbit manifolds and naturally reductive manifolds}\label{conef}

\begin{definition}\label{defg.o.}A geodesic orbit manifold (or g.o. manifold) is a Riemannian manifold $(M,g)$ such that any geodesic is the integral curve of a complete Killing vector field. Equivalently, there exists a Lie group $G$ of isometries of $(M,g)$ such that any geodesic $\gamma$ on $M$ has the form $\gamma(t)=\exp(tX)\cdot p$, $t\in \mathbb R$, where $\exp$ is the Lie exponential map on $G$, $p=\gamma(0)\in M$ and $\cdot$ denotes the (isometric) action of $G$ on $M$.   
\end{definition}

From the above definition, it follows that any connected g.o manifold $(M,g)$ is complete and the Lie group $G$ acts transitively on $M$.  Therefore, $M$ is a homogeneous space $G/K$, where $K$ is the isotropy subgroup of a point $p\in M$. At the same time, the metric $g$ is a $G$-invariant metric on $M=G/K$.  The corresponding Riemannian homogeneous space $(G/K,g)$ is called a \emph{geodesic orbit space} (or \emph{g.o. space}) and the metric $g$ is called a \emph{geodesic orbit metric} (or \emph{g.o. metric}).  To emphasize the Lie group $G$ where necessary, we will call $(M,g)$ a \emph{$G$-g.o. manifold} and the metric $g$ a \emph{$G$-g.o. metric}. 

 \begin{remark}\label{iinvariant}Definition \ref{defg.o.} implies that the property to be a $G$-g.o. manifold is local in the following sense: If $G$ and $G^{\prime}$ are connected Lie groups acting transitively and isometrically on $(M,g)$, and such that the Lie algebras of $G$ and $G^{\prime}$ coincide, then $(M,g)$ is a $G$-g.o. manifold if and only if it is a $G^{\prime}$-g.o. manifold. \end{remark}

 From Definition \ref{defg.o.}, we immediately obtain the following.   

\begin{prop}\label{g.o.1prop}Let $(M,g)$ be a homogeneous Riemannian manifold and write $M$ as the homogeneous space $M=I/K$, where $I$ is the isometry group of $(M,g)$ and $K$ is the isotropy subgroup of the action of $I$ on $M$. Then $(M,g)$ is a g.o. manifold if and only if the Riemannian space $(I/K,g)$ is a g.o. space.  Equivalently, $(M,g)$ is a g.o. manifold if and only if the Riemannian metric $g$ is an $I$-g.o. metric.\end{prop} 

If $M$ is connected, the identity component $I^0$ of $I$ also acts transitively on $M$ and we may replace $I$ with $I^0$ in the above proposition. 

\begin{prop}\label{stranger} A connected homogeneous Riemannian manifold $(M,g)$ is a g.o. manifold if and only if the metric $g$ is an $I^0$-g.o. metric, where $I^0$ is the identity component of the isometry group of $(M,g)$.\end{prop}

The well-known \emph{geodesic lemma} of Kowalski and Vanhecke in \cite{KoVa} asserts that a Riemannian space $(G/K,g)$, with a reductive decomposition $\fr{g}=\fr{k}\oplus \fr{m}$, is a g.o. space if and only for any $X\in \fr{m}$ there exists a vector $W=W(X)\in \fr{k}$ such that 

\begin{equation}\label{geodesiclemma}\langle [W+X,Y]_{\fr{m}},X\rangle=0 \ \ \makebox{for all} \ \ Y\in \fr{m},\end{equation}

\noindent where $\langle \ ,\ \rangle$ is the inner product on $\fr{m}$ corresponding to the metric $g$ and $[W+X,Y]_{\fr{m}}$ denotes the orthogonal projection of $[W+X,Y]$ on $\fr{m}$. \\

The prime examples of g.o. manifolds are the naturally reductive manifolds, defined as follows.

\begin{definition}\label{defnat}
A Riemannian homogeneous space $(G/K,g)$ is called naturally reductive if there exists a reductive decomposition $\fr{g}=\fr{k}\oplus \fr{m}$ such that the inner product $\langle \ ,\ \rangle$ on $\fr{m}$, corresponding to the metric $g$, satisfies 

\begin{equation}\label{natredcond}\langle [X,Y]_{\fr{m}},X\rangle=0 \ \ \makebox{for all} \ \ X,Y\in \fr{m},\end{equation}

\noindent where $[X,Y]_{\fr{m}}$ denotes the orthogonal projection of $[X,Y]$ on $\fr{m}$.  The corresponding metric $g$ is also called naturally reductive.

A Riemannian homogeneous manifold $(M,g)$ is called naturally reductive (or $G$-naturally reductive) if there exists a Lie group $G$ of isometries acting transitively on $M$ such that the corresponding Riemannian homogeneous space $(M=G/K,g)$ is naturally reductive.
\end{definition}  

 \begin{remark}\label{iinvariant2}Similarly with Remark \ref{iinvariant}, the property to be a $G$-naturally reductive manifold is local.  This follows from the fact that condition \eqref{natredcond} is purely of Lie-algebraic nature.  More specifically, if $G$ and $G^{\prime}$ are connected Lie groups acting transitively and isometrically on $(M,g)$, and such that the Lie algebras of $G$ and $G^{\prime}$ coincide, then $(M,g)$ is a $G$-naturally reductive manifold if and only if it is a $G^{\prime}$-naturally reductive manifold. \end{remark}

Comparing Equation \eqref{geodesiclemma} for $W=0$ with Equation \eqref{natredcond}, it becomes evident that any naturally reductive space is a g.o. space and hence any naturally reductive manifold is a g.o. manifold.  Historically, the first example of a non-naturally reductive g.o. manifold is the generalized Heisenberg group (\cite{Ka}). It is not hard to show that any normal metric on $G/K$ (Definition \ref{defnorm}) is naturally reductive and hence geodesic orbit.  In fact, a well-known result of Kostant (\cite{Kos}, \cite{DaZi}) shows that any naturally reductive metric on a (almost effective) space $G/K$ is essentially induced from a bi-invariant (possibly pseudo-Riemannian) metric of a Lie group acting transitively on $G/K$.  

 The following well-known theorem of D'Atri and Ziller characterizes the left-invariant naturally reductive metrics on compact simple Lie groups.  Firstly, let $G$ be a compact simple Lie group with Lie algebra $\fr{g}$ and let $K$ be a closed connected subgroup of $G$ with Lie algebra $\fr{k}$. Let $Q$ be the negative of the Killing form of $\fr{g}$ and decompose $\fr{k}$ into the $Q$-orthogonal sum $\fr{k}=\fr{z}(\fr{k})\oplus \fr{k}_1\oplus \cdots \oplus \fr{k}_s$, where $\fr{z}(\fr{k})$ is the center of $\fr{k}$ and $\fr{k}_j$, $j=1,\dots ,s$, are the simple ideals of $\fr{k}$.  We have the $Q$-orthogonal decomposition 
 
 \begin{equation}\label{reddecg}\fr{g}=\fr{k}\oplus \fr{m}=\fr{z}(\fr{k})\oplus \fr{k}_1\oplus \cdots \oplus \fr{k}_s\oplus \fr{m}, \end{equation}

\noindent where $\fr{m}$ coincides with the tangent space of $G/K$ at the origin.
\begin{theorem}\label{naturallyDaZi}\emph{(\cite{DaZi})} A left-invariant Riemannian metric $g$ on a compact simple Lie group $G$ is naturally reductive if and only if there exists a closed connected subgroup $K$ of $G$ such that the corresponding inner product on $\fr{g}$ has the form 

\begin{equation}\label{DaZiform}\langle \ ,\ \rangle=\left.( \ , \ )\right|_{\fr{z}(\fr{k}) \times \fr{z}(\fr{k})}+\left.\lambda_1 Q\right|_{\fr{k}_1\times \fr{k}_1}+\cdots +\left.\lambda_s Q\right|_{\fr{k}_s\times \fr{k}_s}+\left.\lambda Q\right|_{\fr{m}\times \fr{m}},\end{equation}

\noindent with respect to the decomposition \eqref{reddecg}, where $\left.( \ , \ )\right|_{\fr{z}(\fr{k}) \times \fr{z}(\fr{k})}$ denotes an arbitrary inner product on $\fr{z}(\fr{k})$ and $\lambda_1,\cdots, \lambda_s,\lambda>0$.\end{theorem}

We remark that all the above metrics are $G\times K$-naturally reductive.  More specifically, the metrics are $G\times K$-invariant.  In matrix form, the corresponding metric endomorphism $\Lambda=\Lambda_Q\in \op{End}(\fr{g})$ of the naturally reductive metrics \eqref{DaZiform} is 

\begin{equation*}\label{metricendomorphismnaturally}\Lambda=\begin{pmatrix} 
 \left.\Lambda\right|_{\fr{z}(\fr{k})} & 0 & \cdots &0 &0\\
 0& \left.\lambda_1\op{Id}\right|_{\fr{k}_1} &\cdots &0 &0\\
  \vdots & \cdots & \ddots &\vdots &\vdots\\
  0&\cdots &\cdots &\left.\lambda_s\op{Id}\right|_{\fr{k}_s} &0\\
  0&\cdots &\cdots &0 &\left.\lambda\op{Id}\right|_{\fr{m}}
  \end{pmatrix}.
\end{equation*}

Moreover, the bi-invariant metrics on (compact) Lie groups have the following form (see for example \cite{DaZi}).  

\begin{prop}\label{DaZiBi} A metric $g$ on a compact Lie group $K$ is bi-invariant if and only if the corresponding metric endomorphism $\Lambda\in \op{End}(\fr{k})$ has the form

\begin{equation*}\label{XCD1}\Lambda=\begin{pmatrix} 
 \left.\Lambda\right|_{\fr{z}(\fr{k})} & 0 & \cdots &0\\
 0& \left.\lambda_1\op{Id}\right|_{\fr{k}_1} &\cdots &0\\
  \vdots & \cdots & \ddots &\vdots\\
  0&\cdots &\cdots &\left.\lambda_s\op{Id}\right|_{\fr{k}_s}
  \end{pmatrix},\end{equation*}

\noindent i.e. $\fr{m}$ is trivial. 
\end{prop}

For the following results, we recall that the universal cover of a homogeneous space $G/K$ with $G$ semisimple is the homogeneous space $\widetilde{G}/\widetilde{K}$, where $\widetilde{G}$ is the universal covering group of $G$ and $\widetilde{K}$ is the identity component of $\pi^{-1}(K)$, with $\pi:\widetilde{G}\rightarrow G$ being the canonical projection (\cite{No}).  We will use the following proposition in the sequel.

\begin{prop}\label{UnivCover}\emph{( \cite{SoDed}, \cite{ArSoSt1})} If any $\widetilde{G}$-invariant g.o. metric on the universal cover $\widetilde{G}/\widetilde{K}$ of a homogeneous space $G/K$ is $\widetilde{G}$-normal (resp. standard if $G$ is simple) then any $G$-invariant g.o. metric on $G/K$ is $G$-normal (resp. standard). \end{prop}

Similarly, we have the following.

\begin{prop}\label{UnivCover1} If any $\widetilde{G}$-invariant g.o. metric on the universal cover $\widetilde{G}/\widetilde{K}$ of a homogeneous space $G/K$ is $\widetilde{G}$-naturally reductive then any $G$-invariant g.o. metric on $G/K$ is $G$-naturally reductive. \end{prop}

\begin{proof}Let $\fr{g}$ be the Lie algebra of the groups $G$, $\widetilde{G}$ and let $\fr{k}$ be the Lie algebra of the groups $K$, $\widetilde{K}$.  For any reductive decomposition $\fr{g}=\fr{k}\oplus \fr{m}$, the subspace $\fr{m}$ coincides with both tangent spaces $T_o(G/K)$ and $T_o(\widetilde{G}/\widetilde{K})$. Let $g$ be a $G$-invariant g.o. metric on $G/K$. Then $g$ corresponds to an $\op{Ad}_K$-invariant inner product $\langle \ ,\ \rangle$ on $\fr{m}$, which is also $\op{Ad}_{\widetilde{K}}$-invariant given that $\widetilde{K}$ is connected.  Therefore, the product $\langle \ ,\ \rangle$, and hence the metric $g$, defines a $\widetilde{G}$-invariant metric on $\widetilde{G}/\widetilde{K}$. Since $g$ is a g.o. metric on $G/K$, the corresponding inner product $\langle \ ,\ \rangle$ satisfies Equation \eqref{geodesiclemma}, and thus the same equation implies that $g$ is also a g.o. metric on $\widetilde{G}/\widetilde{K}$. By the assumption of the proposition, $g$ is a $\widetilde{G}$-naturally reductive metric on $\widetilde{G}/\widetilde{K}$.  Therefore, the product $\langle \ ,\ \rangle$ satisfies Equation \eqref{natredcond} for some choice of the orthogonal complement $\fr{m}$.  By the same equation, it follows that the metric $g$ is $G$-naturally reductive on $G/K$.\end{proof}

\subsection{Riemannian homogeneous spaces and Lie groups as geodesic orbit manifolds}

Let $G/K$ be a homogeneous space with $G$ compact and consider the reductive decomposition $\fr{g}=\fr{k}\oplus \fr{m}$ with respect to a fixed $\op{Ad}$-invariant inner product $Q$ on $\fr{g}$.  The following is a sufficient and necessary algebraic condition for a metric endomorphism $\Lambda=\Lambda_Q$, given by Equation \eqref{metend1}, to define a g.o. metric on $G/K$.

\begin{prop}\emph{(\cite{AlAr}, \cite{SoTr})}\label{GOCond} A $G$-invariant metric on $G/K$, with corresponding metric endomorphism $\Lambda\in \op{End}(\fr{m})$, is $G$-geodesic orbit if and only if for any vector $X\in \fr{m}$ there exists a vector $W=W(X)\in \fr{k}$ such that  

\begin{equation*}\label{cor}[W+X,\Lambda X]=0.\end{equation*}
\end{prop}

The following lemma, which we call the \emph{normalizer lemma}, is fundamental for the study of g.o. spaces.

\begin{lemma}\label{NormalizerLemma}\emph{(\cite{NikAn})} 
Let $(G/K,g)$ be a geodesic orbit Riemannian space.  Then the inner product $\langle \ ,\ \rangle$ generating the metric $g$ is not only $\op{Ad}_K$-invariant but also $\op{Ad}_{N^0(K)}$-invariant, where $N^0(K)$ is the identity component of the normalizer of $K$ in $G$.\end{lemma} 

An important consequence of the above lemma is the following diagonalization result (see for example \cite{SoDed}).

\begin{corol}\label{nikoniko}Let $(G/K,\Lambda)$ be a g.o. space with $G$ compact and $K$ connected.  Consider the $Q$-orthogonal decomposition 

\begin{equation*}\fr{m}=\fr{c}_{\fr{m}}(\fr{k})\oplus \fr{p},\end{equation*}

\noindent of the tangent space $\fr{m}$ of $G/K$ at the origin, where $\fr{c}_{\fr{m}}(\fr{k})$ is the Lie algebra of $N^0(K)/K$ and $\fr{p}$ is the tangent space of $G/N^0(K)$ at the origin.  Then the metric $\Lambda$ admits the block-diagonal form

\begin{equation*}\Lambda=\begin{pmatrix} \left.\Lambda\right|_{\fr{c}_{\fr{m}}(\fr{k})} & 0\\
0 &  \left.\Lambda\right|_{\fr{p}}  \end{pmatrix},\end{equation*}

\noindent where $\left.\Lambda\right|_{\fr{c}_{\fr{m}}(\fr{k})}$ defines a bi-invariant metric on the group $N^0(K)/K$ and $\left.\Lambda\right|_{\fr{p}}$ defines a $G$-g.o. metric on the homogeneous space $G/N^0(K)$.\end{corol}

Let $K$ be a closed subgroup of $G$ with Lie algebra $\fr{k}$.  For the rest of this section we restrict our attention to $G\times K$-invariant metrics on $G$. 

\begin{prop}\label{GoConditionBasic} A $G\times K$-invariant metric $g$ on $G$, with corresponding metric endomorphism $\Lambda\in \op{End}(\fr{g})$, is $G\times K$-geodesic orbit if and only if for any vector $X\in \fr{g}$ there exists a vector $W=W(X)\in \fr{k}$ such that 

\begin{equation*}\label{mpeee}[W+X,\Lambda X]=0.\end{equation*}\end{prop}

 The above condition was firstly presented in \cite{NikEin} for compact simple Lie groups and under the assumption that $\fr{k}$ is an \emph{adapted} subalgebra.  Here we provide a proof that bypasses both the simplicity assumption and the condition for $\fr{k}$ to be adapted.  Our proof uses results from \emph{two-step geodesic orbit spaces}, which are defined as those homogeneous spaces whose geodesics are orbits of exponential factors of the form $\gamma(t)=\exp(tZ)\exp(tW)$ (\cite{ArCalSo}, \cite{ArSo}).  \\

\noindent \emph{Proof of Proposition \ref{GoConditionBasic}.}  The group $G\times K$ acts on $G$ by $(x,y)\cdot z:=xzy^{-1}$, where $x,z\in G$, $y\in K$. Recall that the Lie exponential map $\exp_{\fr{g}\times\fr{k}}:\fr{g}\times \fr{k}\rightarrow G\times K$ satisfies $\exp_{\fr{g}\times\fr{k}}(Z,W)=(\exp(Z),\exp(W))$ for $(Z,W)\in \fr{g}\times\fr{k}$, where $\exp:\fr{g}\rightarrow G$ denotes the exponential map on $G$.  On the other hand, $(G,g)$ is a $G\times K$-g.o. manifold if and only if every geodesic $\gamma$ on $G$ has the form $\gamma(t)=\exp_{\fr{g}\times\fr{k}}\big( t(Z,W)\big)\cdot z$. Since the action of $G\times K$ on $G$ is transitive, we lose no generality in assuming $z=e$.  Therefore, and in view of the discussion above, $(G,g)$ is a $G\times K$-g.o. manifold if and only if every geodesic $\gamma$ on $G$ through $e$ has the form

\begin{equation}\label{twostep}\gamma(t)=\exp_{\fr{g}\times\fr{k}}\big( t(Z,W)\big)\cdot e=\exp(tZ)e\exp^{-1}(tW)=\exp(tZ)\exp(-tW).\end{equation}

\noindent Observe that the above curve satisfies $\gamma(0)=e$ and $\dot{\gamma}(0)=Z-W$. We set $X:=Z-W=\dot{\gamma}(0)$.  It follows that $(G,g)$ is a $G\times K$-g.o. manifold if and only if for any $X\in \fr{g}$, there exists a $W\in \fr{k}$ such that the unique geodesic $\gamma$ with $\gamma(0)=e$ and $\dot{\gamma}(0)=X$ has the form \eqref{twostep} with $Z-W=X$.  

Since $W\in \fr{k}$ and the metric $g$ is $G\times K$-invariant, Lemma \ref{eigenspaceinvariant1} implies that the endomorphism $\op{ad}_W$ of $\fr{g}$ is skew-symmetric with respect to the inner product $\langle \ ,\ \rangle$ corresponding to $g$.  Then Proposition 4.3. for $k=0$ in \cite{ArCalSo} asserts that the curve \eqref{twostep} is a geodesic if and only if 

\begin{equation*}[Z-W,\Lambda(Z-W)]-\Lambda[Z,W]=0.\end{equation*}

\noindent Recalling that $X=Z-W=\dot{\gamma}(0)$, observing that $[Z,W]=[X,W]$ and taking into account the $\op{ad}_{\fr{k}}$-equivariance of $\Lambda$ along with the fact that $W\in\fr{k}$, the above equation is equivalent to 

\begin{equation}\label{concludreal}0=[X,\Lambda X]-\Lambda[X,W]=[X,\Lambda X]-[\Lambda X,W]=[W+X,\Lambda X].\end{equation}

We conclude that $(G,g)$ is a $G\times K$-g.o. manifold if and only if for any $X\in \fr{g}$ the unique geodesic $\gamma$ with $\gamma(0)=e$ and $\dot{\gamma}(0)=X$ has the form \eqref{twostep} with $Z-W=X$, i.e if and only if for any $X\in \fr{g}$ there exists a $W\in \fr{k}$ such that Equation \eqref{concludreal} is true. \qed \\

We close this section with an analogue to the normalizer Lemma \ref{NormalizerLemma} for $G\times K$-g.o. metrics. Denote by 

\begin{equation*}\fr{n}_{\fr{g}}(\fr{k})=\{X\in \fr{g}:[X,Y]\in \fr{k}\ \ \makebox{for all} \ \ Y\in \fr{k}\},\end{equation*}

\noindent  the normalizer of $\fr{k}$ in $\fr{g}$. If $K$ is connected, $\fr{n}_{\fr{g}}(\fr{k})$ coincides with the Lie algebra of $N^0(K)$.  

\begin{lemma}\label{NormalizerLemmaNew} 
Let $G$ be a compact Lie group and let $\Lambda\in \op{End}(\fr{g})$ be the metric operator corresponding to a $G\times K$-geodesic orbit metric on $G$, where $K$ is a closed subgroup of $G$.  Then $\Lambda$ is not only $\op{ad}_{\fr{k}}$-equivariant, but also $\op{ad}_{\fr{n}_{\fr{g}}(\fr{k})}$-equivariant.\end{lemma} 

\begin{proof}  

Let $\langle \ , \ \rangle$ be the inner product on $\fr{g}$ corresponding to the $G\times K$-geodesic orbit metric on $G$.  The group $G\times K$ acts transitively and isometrically on $\big(G, \langle \ , \ \rangle\big)$ by $(x,y)\cdot z:=xzy^{-1}$, where $x,z\in G$, $y\in K$.  Then the geodesic orbit manifold $\big(G, \langle \ , \ \rangle\big)$ is isometric to a homogeneous geodesic orbit space $\big((G\times K)/\Delta K, \rho\big)$ where $\Delta K=\{(k,k)\in G\times K: k\in K$ is the diagonal of $K$ and $\rho$ is an $\op{Ad}_{\Delta K}$-invariant inner product on the tangent space of $(G\times K)/\Delta K$ at the origin. We will prove the lemma by applying Lemma \ref{NormalizerLemma} to the geodesic orbit space $\big((G\times K)/\Delta K, \rho \big)$.   

The $\op{Ad}_G$-invariant inner product $Q$ on $\fr{g}$ extends naturally to an $\op{Ad}_{G\times K}$-invariant inner product $\tilde{Q}$ on the Lie algebra $\fr{g}\oplus \fr{k}$ of $G\times K$. In turn, we have the $\tilde{Q}$-orthogonal decomposition  

\begin{equation*}\fr{g}\oplus \fr{k}=\Delta \fr{k}\oplus \tilde{\fr{m}}, \end{equation*}

\noindent where $\Delta \fr{k}=\{(X,X):X\in \fr{k}\}$ is the diagonal of $\fr{k}$ and $\tilde{\fr{m}}=\{(X,-\pi_{\fr{k}}(X)):X\in \fr{g}\}$, with $\pi_{\fr{k}}:\fr{g}\rightarrow \fr{k}$ being the orthogonal projection.  The space $\tilde{\fr{m}}$ can be identified with the tangent space of $(G\times K)/\Delta K$ at the origin and is clearly isomorphic to the Lie algebra $\fr{g}$ (as vector spaces).  Let $\tilde{X}=(X,-\pi_{\fr{k}}(X))$ and $\tilde{Y}=(Y,-\pi_{\fr{k}}(Y))$ be vectors in  $\tilde{\fr{m}}$, where $X,Y \in \fr{g}$.  Then the $\op{Ad}_{\Delta K}$-invariant inner product $\rho$ on  $\tilde{\fr{m}}$ is given by

\begin{equation}\label{makarioxi}\rho(\tilde{X},\tilde{Y})=\langle X,Y\rangle.\end{equation}

\noindent Lemma \ref{NormalizerLemma}, applied to the geodesic orbit space $\big((G\times K)/\Delta K, \rho\big)$, implies that any operator in $\op{ad}_{\fr{n}_{\fr{g}\oplus \fr{k}}(\Delta \fr{k})}$ is skew-symmetric with respect to the inner product $\rho$, where $\fr{n}_{\fr{g}\oplus \fr{k}}(\Delta \fr{k})$ is the normalizer of $\Delta\fr{k}$ in $\fr{g}\times \fr{k}$.  In other words, for any $(Z,W)\in \fr{n}_{\fr{g}\oplus \fr{k}}(\Delta \fr{k})$ and for all $X\in \fr{g}$, Equation \eqref{makarioxi} for the vector $\tilde{X}=(X,-\pi_{\fr{k}}(X))\in \tilde{\fr{m}}$ yields

 \begin{equation}\label{xloup}0=\rho\big(\op{ad}_{(Z,W)}\tilde{X},\tilde{X}\big)=\rho\big((\op{ad}_ZX,-\op{ad}_W\pi_{\fr{k}}(X)),(X,-\pi_{\fr{k}}(X))\big)=\langle \op{ad}_ZX,X\rangle.\end{equation}

  Now consider the $Q$-orthogonal sum (see for example \cite{NikAn}) 
 
 \begin{equation}\label{sum1norm}\fr{n}_{\fr{g}}(\fr{k})=\fr{k}_s\oplus \fr{c}_{\fr{g}}(\fr{k}), \ \ \makebox{where} \ \ \fr{c}_{\fr{g}}(\fr{k})=\{X\in \fr{g}:[X,Y]=0\ \ \makebox{for all} \ \ Y\in \fr{k}\}\end{equation}

\noindent is the centralizer of $\fr{k}$ in $\fr{g}$ and $\fr{k}_s$ is the semisimple part of $\fr{k}$. It is not hard to verify that

\begin{equation}\label{makarioxi2}\fr{n}_{\fr{g}\oplus \fr{k}}(\Delta \fr{k})=\{(Z,W)\in \fr{g}\oplus \fr{k}: Z-W\in \fr{c}_{\fr{g}}(\fr{k})\}.\end{equation}
    
\noindent Moreover, letting $W$ vary over $\fr{k}_s$, decomposition \eqref{sum1norm} and relation \eqref{makarioxi2} imply that for any vector $Z\in \fr{n}_{\fr{g}}(\fr{k})$, there exists a vector of the form $(Z,W)\in \fr{n}_{\fr{g}\oplus \fr{k}}(\Delta \fr{k})$.  Applying Equation \eqref{xloup} for the arbitrary vector $Z\in \fr{n}_{\fr{g}}(\fr{k})$, we deduce that $\langle \op{ad}_ZX,X\rangle=0$ for all $X\in \fr{g}$ and thus any operator in $\op{ad}_{\fr{n}_{\fr{g}}(\fr{k})}$ is skew-symmetric with respect to the inner product $\langle \ , \ \rangle$, yielding the conclusion of the lemma.  \end{proof}
 
A Corollary of Lemma \ref{NormalizerLemmaNew} is the following.

\begin{corol}\label{yields}Let $(G,g)$ be a compact simple Riemannian Lie group where $g$ is a left-invariant geodesic orbit metric. Let $\fr{g}\times \fr{k}$ be the Lie algebra of the isometry group of $(G,g)$.  Then $\fr{k}$ is a self-normalizing subalgebra of $\fr{g}$.\end{corol} 

\begin{proof} By Lemma \ref{NormalizerLemmaNew}, the metric endomorphism $\Lambda$ of $g$ is $\op{ad}_{\fr{n}_{\fr{g}}(\fr{k})}$-equivariant. Equivalently, any operator in $\op{ad}_{\fr{n}_{\fr{g}}(\fr{k})}$ is skew-symmetric with respect to the corresponding inner product $\langle \ ,\ \rangle$ of $g$.  Since $\fr{k}$ is the maximal subalgebra of $\fr{g}$ such that any operator $\op{ad}_X$, $X\in \fr{k}$, is skew-symmetric with respect to $\langle \ ,\ \rangle$ (see for example the proof of Theorem 7 in \cite{DaZi}), it follows that $\fr{k}=\fr{n}_{\fr{g}}(\fr{k})$.\end{proof}

\section{Regular subgroups and their properties}\label{sec3}

Let $G$ be a compact Lie group with Lie algebra $\fr{g}$ and let $K$ be a closed connected regular subgroup of $G$ (Definition \ref{defregular}) with Lie algebra $\fr{k}$.  We fix an $\op{Ad}$-invariant inner product $Q$ on $\fr{g}$. Let $\fr{n}_{\fr{g}}(\fr{k})$ be the normalizer of $\fr{k}$ in $\fr{g}$ and consider the $Q$-orthogonal reductive decomposition

\begin{equation}\label{E3}\fr{g}= \fr{n}_{\fr{g}}(\fr{k})\oplus \fr{p}.\end{equation}

 \noindent Here $\fr{p}$ coincides with the tangent space $T_o(G/N^0(K))$, where $N^0(K)$ is the identity component of the normalizer of $K$ in $G$. We have the following property.

\begin{lemma}\label{key} Let $K$ be a regular subgroup of $G$ and let $\fr{n}_{\fr{g}}(\fr{k})$ be the normalizer of $\fr{k}$ in $\fr{g}$.  Then any non-zero $\op{ad}_{\fr{n}_{\fr{g}}(\fr{k})}$-submodule of $\fr{n}_{\fr{g}}(\fr{k})$ is inequivalent to any non-zero $\op{ad}_{\fr{n}_{\fr{g}}(\fr{k})}$-submodule of $\fr{p}$.
\end{lemma}

\begin{proof} If $\fr{n}_{\fr{g}}(\fr{k})=\{0\}$, the result holds trivially.  Assume that $\fr{n}_{\fr{g}}(\fr{k})\neq \{0\}$ and suppose on the contrary that there exist equivalent non-zero $\op{ad}_{\fr{n}_{\fr{g}}(\fr{k})}$-submodules $\fr{n}\subseteq \fr{n}_{\fr{g}}(\fr{k})$ and $\fr{q}\subseteq \fr{p}$. In other words, there exists an $\op{ad}_{\fr{n}_{\fr{g}}(\fr{k})}$-equivariant isomorphism $\phi:\fr{n}\rightarrow \fr{q}$.  Let $X\in \fr{n}\subseteq \fr{n}_{\fr{g}}(\fr{k})$ be a non-zero vector and let $\phi(X)$ be its non-zero image in $\fr{p}$. Moreover, let $\fr{t}$ be a Cartan subalgebra of $\fr{n}_{\fr{g}}(\fr{k})$ such that $X\in \fr{t}$.  Since $K$ is a regular subgroup of $G$, $\fr{n}_{\fr{g}}(\fr{k})$ has maximal rank and hence $\fr{t}$ is also a Cartan subalgebra of $\fr{g}$. 

Now let $H\in \fr{t}\subseteq \fr{n}_{\fr{g}}(\fr{k})$ be arbitrary.  The $\op{ad}_{\fr{n}_{\fr{g}}(\fr{k})}$-equivariance of $\phi$, along with the fact that $H,X\in \fr{t}$, imply that $[H, \phi(X)]=\phi\big([H,X]\big)=\phi(0)=0$. Since $H$ is arbitrary, we deduce that 

\begin{equation}\label{sfaira3}[\fr{t}, \phi(X)]=0.\end{equation}

\noindent  On the other hand, $\phi(X)$ is a non-zero vector in $\fr{p}$ which, by relation \eqref{E3}, is orthogonal to $\fr{t}\subseteq \fr{n}_{\fr{g}}(\fr{k})$.  From Equation \eqref{sfaira3}, we conclude that $\fr{t}\oplus \op{span}_{\mathbb R}\{\phi(X)\}$ is an abelian subalgebra of $\fr{g}$.  The last statement yields a contradiction, given that $\fr{t}$ is a maximal abelian subalgebra of $\fr{g}$. \end{proof}

Combined with Lemma \ref{equiv}, Lemma \ref{key} is fundamental in simplifying the necessary form of the $G\times K$-g.o. metrics on Lie groups and thus proving our main results.  This is reflected in the following proof of Theorem \ref{split}, which we firstly restate in terms of metric endomorphisms.

\begin{theorem}\label{split2}Let $G$ be a compact Lie group with Lie algebra $\fr{g}$ and let $K$ be a connected regular subgroup of $G$ with Lie algebra $\fr{k}$.  Consider an orthogonal decomposition $\fr{g}=\fr{n}_{\fr{g}}(\fr{k})\oplus \fr{p}$ with respect to an $\op{Ad}$-invariant inner product $Q$ on $\fr{g}$.  Then the metric endomorphism $\Lambda=\Lambda_Q\in\op{End}(\fr{g})$ of any $G\times K$-g.o. metric $g$ on $G$ admits the block-diagonal form

\begin{equation*}\Lambda=\begin{pmatrix} \left.\Lambda\right|_{\fr{n}_{\fr{g}}(\fr{k})} & 0\\
0 & \left.\Lambda\right|_{\fr{p}}\end{pmatrix},\end{equation*}

\noindent with respect to the decomposition $\fr{g}=\fr{n}_{\fr{g}}(\fr{k})\oplus \fr{p}$, where the block $\left.\Lambda\right|_{\fr{n}_{\fr{g}}(\fr{k})}\in \op{End}(\fr{n}_{\fr{g}}(\fr{k}))$ defines a bi-invariant metric on the group $N^0(K)$ and the block $\left.\Lambda\right|_{\fr{p}}\in \op{End}(\fr{p})$ defines a $G$-g.o. metric on the homogeneous space $G/N^0(K)$. \end{theorem}

\begin{proof} From the $Q$-orthogonal reductive decomposition $\fr{g}=\fr{n}_{\fr{g}}(\fr{k})\oplus \fr{p}$, it follows that both spaces $\fr{n}_{\fr{g}}(\fr{k})$ and $\fr{p}$ are $\op{ad}_{\fr{n}_{\fr{g}}(\fr{k})}$-invariant.  Moreover, Lemma \ref{key} asserts that any non-zero $\op{ad}_{\fr{n}_{\fr{g}}(\fr{k})}$-submodule of $\fr{n}_{\fr{g}}(\fr{k})$ is inequivalent to any non-zero $\op{ad}_{\fr{n}_{\fr{g}}(\fr{k})}$-submodule of $\fr{p}$.  Finally, Lemma \ref{NormalizerLemmaNew} implies that the operator $\Lambda$ is $\op{ad}_{\fr{n}_{\fr{g}}(\fr{k})}$-equivariant.  Applying Lemma \ref{equiv}, we deduce that the $\op{ad}_{\fr{n}_{\fr{g}}(\fr{k})}$-equivariant endomorphism $\Lambda$ leaves both spaces $\fr{n}_{\fr{g}}(\fr{k})$ and $\fr{p}$ invariant.  Therefore, $\Lambda$ admits the block-diagonal form 
 
 \begin{equation*}\Lambda=\begin{pmatrix} \left.\Lambda\right|_{\fr{n}_{\fr{g}}(\fr{k})} & 0\\
0 & \left.\Lambda\right|_{\fr{p}}\end{pmatrix}.\end{equation*}

The endomorphism $\left.\Lambda\right|_{\fr{n}_{\fr{g}}(\fr{k})}$ is $\op{ad}_{\fr{n}_{\fr{g}}(\fr{k})}$-equivariant and thus it defines a bi-invariant metric on $N^0(K)$. Moreover, the $\op{ad}_{\fr{n}_{\fr{g}}(\fr{k})}$-equivariant endomorphism $\left.\Lambda\right|_{\fr{p}}$ defines a $G$-invariant metric on the homogeneous space $G/N^0(K)$.  We will show that $\left.\Lambda\right|_{\fr{p}}$ is a $G$-g.o. metric on $G/N^0(K)$.  Indeed, since $\Lambda$ defines a $G\times K$-g.o. metric on $G$, Proposition \ref{GoConditionBasic} implies that for any $X\in \fr{p}\subseteq \fr{g}$ there exists a vector $W=W(X)\in \fr{k}\subseteq \fr{n}_{\fr{g}}(\fr{k})$ such that 

\begin{equation*}0=[W+X,\Lambda X]=[W+X,\left.\Lambda\right|_{\fr{p}} X].\end{equation*}

\noindent In view of Proposition \ref{GOCond}, the above equation implies that the $\op{ad}_{\fr{n}_{\fr{g}}(\fr{k})}$-equivariant endomorphism $\left.\Lambda\right|_{\fr{p}}$ defines a $G$-g.o. metric on $G/N^0(K)$. \end{proof}

From Proposition \ref{DaZiBi}, it follows that the bi-invariant metric $\left.\Lambda\right|_{\fr{n}_{\fr{g}}(\fr{k})}$ has the form

\begin{equation*}\label{XCD1}\left.\Lambda\right|_{\fr{n}_{\fr{g}}(\fr{k})}=\begin{pmatrix} 
 \left.\Lambda\right|_{\fr{z}(\fr{n}_{\fr{g}}(\fr{k}))} & 0 & \cdots &0\\
 0& \left.\lambda_1\op{Id}\right|_{\fr{n}_1} &\cdots &0\\
  \vdots & \cdots & \ddots &\vdots\\
  0&\cdots &\cdots &\left.\lambda_s\op{Id}\right|_{\fr{n}_s}
  \end{pmatrix},\end{equation*}
  
  \noindent where $\fr{z}(\fr{n}_{\fr{g}}(\fr{k}))$ is the center of $\fr{n}_{\fr{g}}(\fr{k})$ and $\fr{n}_j$, $j=1,\dots, s$, are the simple ideals of $\fr{n}_{\fr{g}}(\fr{k})$. Hence we arrive to the following conclusion of Theorem \ref{split2}.

\begin{corol}\label{ulcorol}
Let $G$ be a compact Lie group with Lie algebra $\fr{g}$ and let $K$ be a regular subgroup of $G$ with Lie algebra $\fr{k}$.  Consider an orthogonal decomposition $\fr{g}=\fr{n}_{\fr{g}}(\fr{k})\oplus \fr{p}$ with respect to an $\op{Ad}$-invariant inner product $Q$ on $\fr{g}$.  Then the metric endomorphism $\Lambda=\Lambda_Q\in\op{End}(\fr{g})$ of any $G\times K$-g.o. metric $g$ on $G$ admits the block-diagonal form

\begin{equation*}\label{metricendomorphismnaturally1}\Lambda=\begin{pmatrix} 
 \left.\Lambda\right|_{\fr{z}(\fr{n}_{\fr{g}}(\fr{k}))} & 0 & \cdots &0 &0\\
 0& \left.\lambda_1\op{Id}\right|_{\fr{n}_1} &\cdots &0 &0\\
  \vdots & \cdots & \ddots &\vdots &\vdots\\
  0&\cdots &\cdots &\left.\lambda_s\op{Id}\right|_{\fr{n}_s} &0\\
  0&\cdots &\cdots &0 &\left.\Lambda\right|_{\fr{p}}
  \end{pmatrix},
\end{equation*}

\noindent where the block $\left.\Lambda\right|_{\fr{p}}$ defines a $G$-g.o. metric on the homogeneous space $G/N^0(K)$.\end{corol}

\section{Proof of Theorem \ref{consequence}}\label{sec4.5}

\subsection{Some preliminary results}

Before we prove Theorem \ref{consequence}, we need the following three preliminary results. 

\begin{lemma}\label{aux}Let $(G,g)$ be a connected compact Riemannian simple Lie group and assume that the metric $g$ is $G\times K$-invariant, where $K$ is a connected Lie subgroup of $G$. Then the isometry group of $(G,g)$ is locally isomorphic to $G\times K^{\prime}$, where $K^{\prime}$ is a connected closed subgroup of $G$ such that $K\subseteq K^{\prime}$.\end{lemma}

\begin{proof} Let $\fr{k}$ be the Lie algebra of $K$.  As a result of the Ochiai-Takahashi Theorem (\cite{Ot}), the Lie algebra of the connected isometry group $I^0$ of $(G,g)$ has the form $\fr{g}\times \fr{k}^{\prime}$, where $\fr{k}^{\prime}$ is a subalgebra of $\fr{g}$.  Let $K^{\prime}$ be the connected subgroup of $G$ with Lie algebra $\fr{k}^{\prime}$ so that $I^0$ is locally isomorphic to $G\times K^{\prime}$.  The Lie algebra $\fr{k}^{\prime}$ is the maximal subalgebra of $\fr{g}$ such that any operator $\op{ad}_X$, $X\in \fr{k}^{\prime}$, is skew-symmetric with respect to the corresponding inner product $\langle \ ,\ \rangle$ to $g$.  Since the metric $g$ is $G\times K$-invariant, it follows that $G\times K$ acts by isometries on $(G,g)$ and thus any operator $\op{ad}_X$, $X\in \fr{k}$, is skew-symmetric with respect to $\langle \ ,\ \rangle$ (Lemma \ref{eigenspaceinvariant1}).  By the maximality of $\fr{k}^{\prime}$, we conclude that $\fr{k}\subseteq \fr{k}^{\prime}$. Since $K$ is connected, it follows that $K\subseteq K^{\prime}$.   \end{proof}

\begin{lemma}\label{aux1}
Let $(G,g)$ be a connected compact Riemannian simple Lie group and let $G\times K$ be its connected isometry group up to local isomorphism.  Assume that $K$ has maximal rank in $G$ and the universal cover $\widetilde{G}/\widetilde{K}$ of $G/K$ is not $SO(2l+1)/U(l)$ or $Sp(l)/(U(1)\times Sp(l-1))$. Then $(G,g)$ is a geodesic orbit manifold if and only if it is a naturally reductive manifold.
\end{lemma}

\begin{proof} If $(G,g)$ is naturally reductive then it is clearly geodesic orbit.  Conversely, assume that $(G,g)$ is a g.o. manifold.  Since $K$ has maximal rank in $G$, it is a regular subgroup of $G$.  Moreover, the Lie algebra $\fr{k}$ of $K$ is self-normalizing in $\fr{g}$, i.e. $\fr{k}=\fr{n}_{\fr{g}}(\fr{k})$. From Corollary \ref{ulcorol}, it follows that the metric endomorphism $\Lambda=\Lambda_Q\in\op{End}(\fr{g})$ of $g$ admits the block-diagonal form

\begin{equation}\label{metricendomorphismnaturally1}\Lambda=\begin{pmatrix} 
 \left.\Lambda\right|_{\fr{z}(\fr{k})} & 0 & \cdots &0 &0\\
 0& \left.\lambda_1\op{Id}\right|_{\fr{k}_1} &\cdots &0 &0\\
  \vdots & \cdots & \ddots &\vdots &\vdots\\
  0&\cdots &\cdots &\left.\lambda_s\op{Id}\right|_{\fr{k}_s} &0\\
  0&\cdots &\cdots &0 &\left.\Lambda\right|_{\fr{p}}
  \end{pmatrix},
\end{equation}

\noindent where $\fr{z}(\fr{k})$ is the center of $\fr{k}$, $\fr{k}_1,\cdots ,\fr{k}_s$ are the simple ideals of $\fr{k}$, and the block $\left.\Lambda\right|_{\fr{p}}$ defines a $G$-g.o. metric on the homogeneous space $G/N^0(K)=G/K$. 

 Given that $\widetilde{G}/\widetilde{K}$ is not $SO(2l+1)/U(l)$ or $Sp(l)/(U(1)\times Sp(l-1))$, the classification in \cite{AlNik} implies that any $\widetilde{G}$-invariant g.o. metric on $\widetilde{G}/\widetilde{K}$ is standard. From Proposition \ref{UnivCover}, we conclude that any $G$-invariant g.o. metric on $G/K$ is standard. Therefore,

\begin{equation*} \left.\Lambda\right|_{\fr{p}}=\lambda\left.\op{Id}\right|_{\fr{p}}.\end{equation*}

\noindent Substituting the above equation into relation \eqref{metricendomorphismnaturally1} and taking into account Theorem \ref{naturallyDaZi} (for $\fr{m}=\fr{p}$), we conclude that the metric $\Lambda$ is naturally reductive.  Therefore, $(G,g)$ is a naturally reductive manifold.\end{proof} 

Finally, we will need the main result in \cite{CheCheWo}. Before we state the result, consider a compact simple Lie group $G$ and a connected subgroup $K$ of $G$ such that $G/K$ is a \emph{generalized flag manifold}, i.e. $K$ is the centralizer of a torus in $G$. Consider the decomposition 

\begin{equation*}\fr{g}=\fr{k}\oplus \fr{p}=\underbrace{\fr{z}(\fr{k})\oplus \fr{k}_1\oplus \cdots \oplus \fr{k}_s}_{\fr{k}}\oplus \underbrace{\fr{p}_1\oplus \cdots \oplus \fr{p}_l}_{\fr{p}},\end{equation*}

\noindent where $\fr{z}(\fr{k})$ is the center of $\fr{k}$, $\fr{k}_j$, $j=1,\dots ,s$, are the simple ideals of $\fr{k}$ and $\fr{p}_1\oplus \cdots \oplus \fr{p}_l$ is the decomposition of $\fr{p}$ into irreducible and inequivalent $\op{Ad}_K$-submodules.  Since $K$ has maximal rank in $G$, it is a regular subgroup of $G$ and $\fr{n}_{\fr{g}}(\fr{k})=\fr{k}$. It follows from Corollary \ref{ulcorol} that any $G\times K$-g.o. metric $g$ on $G$ (in fact the corresponding metric endomorphism $\Lambda$) has the form \eqref{metricendomorphismnaturally1}.  Moreover, since the submodules $\fr{p}_i$ are irreducible and pairwise inequivalent, decomposition $\fr{p}_1\oplus \cdots \oplus \fr{p}_l$ coincides with the isotypic decomposition of $\left.\op{Ad}\right|_{K}:K\rightarrow \op{Gl}(\fr{p})$. Therefore, $\left.\Lambda \right|_{\fr{p}}$ has the diagonal form $\left.\Lambda \right|_{\fr{p}}=\begin{pmatrix}\mu_1\left.\op{Id}\right|_{\fr{p}_1} &\cdots & 0\\
\vdots & \ddots &\vdots\\
0 & \cdots & \mu_l\left.\op{Id}\right|_{\fr{p}_l}\end{pmatrix}$.  We conclude that the inner product $\langle \ ,\ \rangle$ corresponding to the g.o. metric $g$ has the form

\begin{equation}\label{CheWoform}\langle \ ,\ \rangle=\left.( \ , \ )\right|_{\fr{z}(\fr{k}) \times \fr{z}(\fr{k})}+\left.\lambda_1 Q\right|_{\fr{k}_1\times \fr{k}_1}+\cdots +\left.\lambda_s Q\right|_{\fr{k}_s\times \fr{k}_s}+\left.\mu_1 Q\right|_{\fr{p}_1\times \fr{p}_1}+ \cdots + \left.\mu_l Q\right|_{\fr{p}_l\times \fr{p}_l},\end{equation}

\noindent where $Q$ is the negative of the Killing form of $\fr{g}$ and $\left.( \ , \ )\right|_{\fr{z}(\fr{k}) \times \fr{z}(\fr{k})}$ is any inner product on $\fr{z}(\fr{k})$.  The main result in \cite{CheCheWo} asserts that any g.o. metric of the form \eqref{CheWoform} is $G\times K$-naturally reductive.  On the other hand, the above discussion shows that the metrics \eqref{CheWoform} exhaust the $G\times K$-g.o. metrics on $G$.  Thus the main result in \cite{CheCheWo} can be restated as follows.

\begin{theorem}\label{CheWotheorem} Let $(G,g)$ be a connected compact Riemannian simple Lie group and assume that the metric $g$ is $G\times K$-invariant, where $K$ is a connected subgroup of $G$ such that $G/K$ is a generalized flag manifold.  Then $(G,g)$ is a $G\times K$-g.o. manifold if and only if it is a $G\times K$-naturally reductive manifold.
\end{theorem}

We are ready to proceed to the main proof.

\subsection{Proof of Theorem \ref{consequence}}

 It suffices to assume that $K$ is connected.  Indeed, let $K^0$ be the identity component of $K$. If the metric $g$ is $G\times K$-invariant then it is also $G\times K^0$-invariant, while if $g$ is $G\times K$-g.o. then it is also $G\times K^0$-g.o, (c.f. Remark \ref{iinvariant}).
 
  Assume that $(G,g)$ is $G\times K$-geodesic orbit.  Since the metric $g$ is $G\times K$-invariant, Lemma \ref{aux} asserts that the connected isometry group of $(G,g)$ is locally isomorphic to $G\times K^{\prime}$, where $K^{\prime}$ is a connected closed subgroup of $G$ such that $K\subseteq K^{\prime}$. If $K^{\prime}=G$ then the metric $g$ is bi-invariant and hence Theorem \ref{consequence} follows.  Assume henceforth that $K^{\prime}\subsetneq G$.  The isotropy subgroup of $e$ with respect to the action of $G\times K^{\prime}$ on $G$ is the diagonal $\Delta K^{\prime}=\{(x,x)\in G\times K^{\prime}: x\in K^{\prime}\}$.  Then $G$ is diffeomorphic to the homogeneous space $(G\times K^{\prime})/\Delta K^{\prime}$. 
 
 Since $g$ is $G\times K$-geodesic orbit, Lemma \ref{NormalizerLemmaNew} implies that the corresponding metric endomorphism $\Lambda$ is $\op{ad}_{\fr{n}_{\fr{g}}(\fr{k})}$-equivariant. The Lie algebra $\fr{k}^{\prime}$ of $K^{\prime}$ is the maximal subalgebra of $\fr{g}$ such that $\Lambda$ is $\op{ad}_{\fr{k}^{\prime}}$-equivariant, and hence $\fr{n}_{\fr{g}}(\fr{k})\subseteq \fr{k}^{\prime}$.  On the other hand, since $K$ is regular, $\fr{n}_{\fr{g}}(\fr{k})$ contains a Cartan subalgebra of $\fr{g}$ and hence has maximal rank in $\fr{g}$.  We conclude that $K^{\prime}$ also has maximal rank in $G$. If the universal cover $\widetilde{G}/\widetilde{K^{\prime}}$ of $G/K^{\prime}$ is not $SO(2l+1)/U(l)$ or $Sp(l)/(U(1)\times Sp(l-1))$ then Theorem \ref{consequence} follows from Lemma \ref{aux1}.
 
   Now assume that the universal cover $\widetilde{G}/\widetilde{K^{\prime}}$ of $G/K^{\prime}$ is either $SO(2l+1)/U(l)$ or $Sp(l)/(U(1)\times Sp(l-1))$.  Then $\widetilde{G}/\widetilde{K^{\prime}}$ is a generalized flag manifold (\cite{AlAr}) and hence Theorem \ref{CheWotheorem} implies that any $\widetilde{G}\times \widetilde{K^{\prime}}$-g.o. metric $\mu$ on $\widetilde{G}$ is $\widetilde{G}\times \widetilde{K^{\prime}}$-naturally reductive (we recall that $\widetilde{G}$ denotes the universal covering group of $G$ while $\widetilde{K^{\prime}}$ denotes the identity component of $\pi^{-1}(K^{\prime})$ where $\pi:\widetilde{G}\rightarrow G$ is the canonical projection). 
   
   On the other hand, the universal covering group $\widetilde{G\times K^{\prime}}$ of $G\times K^{\prime}$ acts transitively and isometrically on $\widetilde{G}=\widetilde{G\times K^{\prime}}/\widetilde{\Delta K^{\prime}}$, while the Lie algebra $\fr{g}\oplus \fr{k}^{\prime}$ of $\widetilde{G\times K^{\prime}}$ coincides with the Lie algebra of $\widetilde{G}\times \widetilde{K^{\prime}}$. Remark \ref{iinvariant} then implies that $(\widetilde{G},\mu)$ is a $\widetilde{G\times K^{\prime}}$-g.o. manifold if and only if it is a $\widetilde{G}\times \widetilde{K^{\prime}}$-g.o. manifold.  Similarly, Remark \ref{iinvariant2} implies that $(\widetilde{G},\mu)$ is a $\widetilde{G\times K^{\prime}}$-naturally reductive manifold if and only if it is a $\widetilde{G}\times \widetilde{K^{\prime}}$-naturally reductive manifold.  Given that any $\widetilde{G}\times \widetilde{K^{\prime}}$-invariant g.o. metric $\mu$ on $\widetilde{G}$ is $\widetilde{G}\times \widetilde{K^{\prime}}$-naturally reductive, it follows that any $\widetilde{G\times K^{\prime}}$-invariant g.o. metric $\mu$ on $\widetilde{G}=\widetilde{G\times K^{\prime}}/\widetilde{\Delta K^{\prime}}$ is $\widetilde{G\times K^{\prime}}$-naturally reductive.  Proposition \ref{UnivCover1} then implies that any $G\times K^{\prime}$-g.o. metric on $G=(G\times K^{\prime})/\Delta K^{\prime}$ is $G\times K^{\prime}$-naturally reductive.  It follows that $(G,g)$ is a naturally reductive manifold.

\section{Applications of Theorem \ref{consequence}}\label{sec8}

In this section, we prove the following applications of Theorem \ref{consequence}: Corollary \ref{consequence1}, Theorem \ref{theoremlit1} and Corollary \ref{Emil}. \\

\noindent \emph{Proof of Corollary \ref{consequence1}}. If $(G,g)$ is a naturally reductive manifold then it is a g.o. manifold.  Conversely, assume that $(G,g)$ is a g.o. manifold.  We will prove that it is a naturally reductive manifold.  Since the metric $g$ is $G\times K$-invariant, it is also $G\times K^0$-invariant, where $K^0$ is the identity component of $K$. By Lemma \ref{aux}, the isometry group of $(G,g)$ is locally isomorphic to $G\times K^{\prime}$, where $K^{\prime}$ is a connected closed subgroup of $G$ such that $K^0\subseteq K^{\prime}$.  More specifically, $(G,g)$ is a $G\times K^{\prime}$-g.o. manifold. On the other hand, $K$ and $K^0$ have maximal rank and thus $K^{\prime}$ also has maximal rank. More specifically, $K^{\prime}$ is a regular subgroup of $G$ and $g$ is a $G\times K^{\prime}$-g.o. metric.   By Theorem \ref{consequence}, $(G,g)$ is naturally reductive.\qed \\

We proceed to the proof of Theorem \ref{theoremlit1}.\\

\noindent \emph{Proof of Theorem \ref{theoremlit1}}.   In the works \cite{ArMo}, \cite{ArSaSt1}, \cite{CheChe}, \cite{CheCheArx}, \cite{ChSa}, \cite{Mori} and \cite{ZaCheDe1}, any of the non-naturally reductive, left-invariant Einstein metrics $g$ on the compact simple Lie groups $G$ is $G\times K$-invariant, where $K$ is a subgroup of maximal rank in $G$.  It follows from Corollary \ref{consequence1} that $(G,g)$ is not geodesic orbit.  More specifically in \cite{ArMo}, $G$ is a compact simple Lie group and $K$ is a closed connected subgroup of $G$ such that $G/K$ is a flag manifold with two isotropy summands.  In \cite{ArSaSt1}, $G=Sp(k_1+k_2+k_3)$ and $K=Sp(k_1)\times Sp(k_2)\times Sp(k_3)$. In \cite{CheChe}, $G=Sp(2k+l)$ and $K=Sp(k)\times Sp(k)\times Sp(l)$.  In \cite{CheCheArx}, we have $(G,K)$ being one of the pairs $(Sp(4),4 Sp(1))$,  $(Sp(2k+l), Sp(k)\times Sp(k)\times Sp(l))$.  In \cite{ChSa}, we have various pairs $(G,K)$ where $G$ is a simple exceptional Lie group (see Table 1 in \cite{ChSa}; the case of $G=E_6$ and $K=SU(3)\times SU(2)\times SU(2)\times U(1)$ has a typographic error, the correct subgroup $K=SU(3)\times SU(3)\times SU(2)\times U(1)$ is shown in Table 2).  In \cite{Mori}, $G=SU(n)$, $n\geq 6$, and $K=S(U(p)\times U(q)\times U(r))$ with $p+q+r=n$.  Finally, in \cite{ZaCheDe1} we have $G=Sp(n)$, $n\geq 4$, and $K=Sp(n-3)\times 3 Sp(1)$.

 Lastly, we consider the non-naturally reductive Einstein metrics induced from standard triples $(G,K,H)$ in \cite{YanDen}.  Here $H\subsetneq K\subsetneq G$, where $G$ is a compact connected simple Lie group and $K$ is a closed subgroup such that $G/K$ is an irreducible symmetric space.  Those metrics on $G$ are $G\times H$-invariant and are induced from inner products of the form

\begin{equation}\label{standardtriples}\langle \ ,\ \rangle=\left.Q\right|_{\fr{h}\times\fr{h}}+x \left.Q\right|_{\fr{u}\times\fr{u}}+y\left.Q\right|_{\fr{p}\times\fr{p}},\end{equation}

\noindent where $Q$ is the negative of the Killing form of $\fr{g}$, $\fr{p}$ is the $Q$-orthogonal complement of $\fr{k}$ in $\fr{g}$ and $\fr{u}$ is the $Q$-orthogonal complement of $\fr{h}$ in $\fr{k}$ (here $\fr{k}$, $\fr{h}$ denote the Lie algebras of $K$, $H$ respectively).  As in \cite{YanDen}, we consider the respective Lie algebra triples $(\fr{g},\fr{k},\fr{h})$.

For the standard triples $(\fr{sp}(2n_1n_2), \fr{sp}(n_1n_2)\times \fr{sp}(n_1n_2),2n_2\fr{sp}(n_1))$, $(\fr{f}_4,\fr{so}(9),\fr{so}(8))$,\\
 $(\fr{e}_8,\fr{so}(16),8\fr{su}(2))$ and $(\fr{e}_8,\fr{so}(16),2\fr{so}(8))$, the Lie algebra $\fr{h}$ has maximal rank in $\fr{g}$ and hence the corresponding subgroup $H$ has maximal rank in $G$.  Since the metrics \eqref{standardtriples} are $G\times H$-invariant and non-naturally reductive, it follows from Corollary \ref{consequence1} that they are also not geodesic orbit.  \qed  \\

\noindent \emph{Proof of Corollary \ref{Emil}.} If $(G,g)$ is naturally reductive then it is geodesic orbit.  Conversely, assume that $(G,g)$ is geodesic orbit.  The isometry group of $(G,g)$ is locally isomorphic to $G\times K$, where $K$ is a closed connected subgroup of $G$.  Thus $(G,g)$ is a $G\times K$-g.o. manifold.  On the other hand, any Lie subalgebra of the Lie algebra of $G$ either has maximal rank or is abelian.  Therefore, $K$ is a regular subgroup of $G$.  From Theorem \ref{consequence}, we conclude that $(G,g)$ is a naturally reductive manifold.\qed

\section{Weakly regular subgroups}\label{weaklyregular}

Let $G$ be a compact Lie group with Lie algebra $\fr{g}$, let $K$ be a closed subgroup of $G$ with Lie algebra $\fr{k}$ and recall the decomposition 

\begin{equation*}\fr{g}=\fr{n}_{\fr{g}}(\fr{k})\oplus \fr{p}.\end{equation*}

\noindent  Recall also that if $K$ is a regular subgroup of $G$, then by Lemma \ref{key} any non-zero $\op{ad}_{\fr{n}_{\fr{g}}(\fr{k})}$-submodule of $\fr{n}_{\fr{g}}(\fr{k})$ is inequivalent to any non-zero $\op{ad}_{\fr{n}_{\fr{g}}(\fr{k})}$-submodule of $\fr{p}$. We use this particular property to define weakly regular subgroups.

\begin{definition}\label{weaklyregulardef}We call a Lie subgroup $K$ of a compact Lie group $G$ weakly regular if the Lie algebra $\fr{k}$ of $K$ satisfies the conclusion of Lemma \ref{key}, i.e. if any non-zero $\op{ad}_{\fr{n}_{\fr{g}}(\fr{k})}$-submodule of $\fr{n}_{\fr{g}}(\fr{k})$ is inequivalent to any non-zero $\op{ad}_{\fr{n}_{\fr{g}}(\fr{k})}$-submodule of $\fr{p}$. \end{definition}

Clearly, any regular subgroup is weakly regular.  The converse is not true as the following example demonstrates.

\begin{example}\label{examplee} Let $G$ be the Lie group $SO(n)$ and let $K$ be the Lie subgroup $SO(k_1)\times \cdots \times SO(k_s)$, embedded diagonally in $G$, where $k_1+\cdots +k_s=n$, $s\geq 3$ and $k_i\geq 2$.  Then the subgroup $K$ is not (always) regular in $G$ but it is weakly regular.  

\end{example}

\begin{proof}  For simplicity, we will prove the statement for $s=3$, i.e. $K=SO(k_1)\times SO(k_2) \times SO(k_3)$. The proof for larger values of $s$ is similar. The subgroup $K$ does not always have maximal rank in $G$ (e.g. if $G=SO(9)$ and $K=SO(3)\times SO(3)\times SO(3)$) while it is not hard to show that the Lie algebra $\fr{k}=\fr{so}(k_1)\oplus \fr{so}(k_2)\oplus \fr{so}(k_3)$ of $K$ is self-normalizing in the Lie algebra $\fr{g}=\fr{so}(n)$ of $G$ (see for example \cite{ArSoSt1}).  Therefore, $\fr{k}$ is not necessarily normalized by a Cartan subalgebra of $\fr{g}$ and thus $K$ is not always a regular subgroup of $G$. 

 To see that $K$ is a weakly regular subgroup of $G$, recall firstly that since $\fr{k}$ is self-normalizing in $\fr{g}$ we have $\fr{n}_{\fr{g}}(\fr{k})=\fr{k}$, while $\fr{p}$ coincides with the tangent space of $G/K$ at the origin. In view of Definition \ref{weaklyregulardef}, to show that $K$ is weakly regular in $G$ it suffices to show that any non-zero $\op{ad}_{\fr{k}}$-submodule of $\fr{k}$ is inequivalent to any non-zero $\op{ad}_{\fr{k}}$-submodule of $\fr{p}$.  In fact, it suffices to verify the aforementioned property for the irreducible submodules. 
 
Now the irreducible $\op{ad}_{\fr{k}}$-submodules of $\fr{k}$ are precisely the spaces $\fr{k}_j$, where $\fr{k}_j$ are the simple or one-dimensional ideals of $\fr{k}$. The latter coincide with $\fr{so}(k_j)$ if $2\leq k_j\neq 4$, while $\fr{so}(4)$ decomposes into simple ideals as $\fr{so}(4)=\fr{so}(3)\oplus \fr{so}(3)$. On the other hand, the space $\fr{p}$ decomposes into $\op{ad}_{\fr{k}}$-submodules as $\fr{p}=\fr{m}_{12}\oplus \fr{m}_{13}\oplus \fr{m}_{23}$, where $\fr{m}_{ij}$ are irreducible submodules of dimension $k_ik_j$ or reduce to the sum of two submodules if $k_i=k_j=2$ (see Section 4 in \cite{ArSoSt1}).  The following is a matrix depiction of the (upper triangular part) of the aforementioned decomposition $\fr{n}_{\fr{g}}(\fr{k})\oplus\fr{p}=\fr{k}\oplus \fr{p}$: 

\begin{equation*}\begin{pmatrix} \fr{so}(k_1)&  \fr{m}_{12} & \fr{m}_{13}\\
*& \fr{so}(k_2)& \fr{m}_{23}\\
*& *& \fr{so}(k_3)\end{pmatrix}.\end{equation*}

Any ideal of $\fr{k}=\bigoplus_{i=1}^3\fr{so}(k_i)$ is inequivalent as an $\op{ad}_{\fr{k}}$-submodule to any (irreducible) submodule of $\fr{m}_{ij}$.  To verify this fact choose a $k_i$, say $k_1$.  If $k_1\neq 4$ then $\fr{so}(k_1)$ is simple or one-dimensional.  The ideal $\fr{so}(k_1)$ cannot be equivalent to any submodule of $\fr{m}_{12}$ because $\op{ad}_{\fr{so}(k_2)}\fr{so}(k_1)=\{0\}$ while $\op{ad}_{\fr{so}(k_2)}\fr{m}_{12}=\fr{m}_{12}$ (see also Section 4 in \cite{ArSoSt1}). Similarly, $\fr{so}(k_1)$  cannot be equivalent to $\fr{m}_{13}$ or $\fr{m}_{23}$ because $\fr{so}(k_3)$ has non-zero action on those submodules. The above arguments are still true if $k_1=4$ and $\fr{so}(k_1)=\fr{so}(3)\oplus \fr{so}(3)$.  Repeating the same arguments for all $k_i$, we deduce that any non-zero irreducible $\op{ad}_{\fr{k}}$-submodule of $\fr{k}$ is inequivalent to any non-zero irreducible $\op{ad}_{\fr{k}}$-submodule of $\fr{p}$ and hence the same is true for all submodules.  We conclude that $K$ is a weakly regular subgroup of $G$.  \end{proof}

The following is an example of a subgroup that is not weakly regular.

\begin{example}\label{examplenee}  Consider a direct product $G\times G$, where $G$ is a compact simple Lie group, and consider the diagonal $\Delta G=\{(g,g):g\in G\}$ as a subgroup of $G\times G$.  Then $\Delta G$ is not a weakly regular subgroup of $G\times G$.  Indeed, it is can be easily verified that the Lie algebra $\fr{g}\oplus \fr{g}$ of $G\times G$ decomposes as $\fr{g}\oplus \fr{g}=\Delta \fr{g}\oplus \fr{p}$ (the right - hand side is a direct sum of vector spaces), where $\Delta \fr{g}=\{(X,X)\in \fr{g}\oplus \fr{g}:X\in \fr{g}\}$ is the Lie algebra of $\Delta G$ and $\fr{p}$ is an $\op{ad}_{\Delta\fr{g}}$-invariant space given by $\fr{p}=\{(X,-X)\in \fr{g}\oplus \fr{g}:X\in \fr{g}\}$. Since $\fr{g}$ has trivial center, one may check that the Lie algebra $\Delta\fr{g}$ is self-normalizing in $\fr{g}\oplus \fr{g}$.  Moreover, the linear isomorphism $\phi:\Delta \fr{g}\rightarrow \fr{p}$ that assigns to each vector $(X,X)\in \Delta \fr{g}$ the vector $(X,-X)\in \fr{p}$ is $\op{ad}_{\Delta\fr{g}}$-equivariant because for all $(X,X)$ and $(Y,Y)\in \Delta \fr{g}$ we have

\begin{eqnarray*}\phi\big([(X,X),(Y,Y)]\big)&=&\phi\big(([X,Y],[X,Y])\big)=([X,Y],-[X,Y])=[(X,X),(Y,-Y)]\\
&=&[(X,X),\phi((Y,Y))],\end{eqnarray*}
\noindent i.e. $\phi$ commutes with any operator in $\op{ad}_{\Delta\fr{g}}$.  Therefore, the non-zero submodule $\Delta\fr{g}$ of $\fr{n}_{\fr{g}\times \fr{g}}(\Delta\fr{g})=\Delta\fr{g}$ is equivalent to $\fr{p}$, and thus $\Delta G$ is not a weakly regular subgroup of $G\times G$ according to Definition \ref{weaklyregulardef}.  \end{example}

We proceed to discuss another important category of weakly regular subgroups.  Recall that a connected and effective homogeneous space $G/K$ is called \emph{isotropy irreducible} if $K$ is compact and the tangent space $\fr{p}=T_o(G/K)$ is irreducible under the isotropy representation $K\rightarrow \op{Gl}(\fr{p})$. An isotropy irreducible space $G/K$ is called \emph{strongly isotropy irreducible} if the identity component $K^0$ of $K$ also acts irreducibly on $\fr{p}$.  Up to homothety, any isotropy irreducible space admits a single $G$-invariant metric which is $G$-normal and hence geodesic orbit. Isotropy irreducible spaces have been studied, among other works, in \cite{WaZi} and \cite{Wo}.

\begin{prop}\label{corol000} Let $G$ be a compact simple Lie group and let $K$ be a closed subgroup of $G$ such that $G/K$ is a strongly isotropy irreducible space.  Then $K$ is a weakly regular subgroup of $G$ and the Lie algebra $\fr{k}$ of $K$ is self-normalizing in the Lie algebra $\fr{g}$ of $G$.\end{prop}

\begin{proof}  We consider the reductive decomposition $\fr{g}=\fr{k}\oplus \fr{p}$ of $G/K$. From the irreducibility of $\fr{p}$, it follows that the Lie algebra $\fr{k}$ is maximal in $\fr{g}$ and hence $\fr{n}_{\fr{g}}(\fr{k})=\fr{k}$.  Assume on the contrary that $K$ is not weakly regular.  Then by Definition \ref{weaklyregulardef} and the above observation, there exists a non-zero $\operatorname{ad}_{\fr{k}}$-submodule of $\fr{p}$ that is equivalent to an $\operatorname{ad}_{\fr{k}}$-submodule (i.e. an ideal) of $\fr{k}$. Since $\fr{p}$ is $\operatorname{ad}_{\fr{k}}$-irreducible, it follows that $\fr{p}$ is equivalent to an ideal $\fr{s}$ of $\fr{k}$.  The ideal $\fr{s}$ cannot be abelian, for otherwise, the equivalence with $\fr{p}$ would imply $\operatorname{ad}_{\fr{k}}\fr{p}=0$, a contradiction.  We conclude that $\fr{s}$ is a simple ideal of $\fr{k}$. Any other ideal $\fr{s}^{\prime}$ of $\fr{k}$ commutes with $\fr{s}$ and thus commutes with $\fr{p}$.  Therefore, if $Q$ is the negative of the Killing form of $\fr{g}$ we have $Q([\fr{p},\fr{p}],\fr{s}^{\prime})=0$ and thus $[\fr{p},\fr{p}]\subseteq \fr{s}$.  The last equation implies that $\fr{s}\oplus \fr{p}$ is an ideal of $\fr{g}$ and, since $\fr{g}$ is simple (\cite{Wo}), we conclude that $\fr{k}=\fr{s}$ and $\fr{g}=\fr{s}\oplus \fr{p}$. Since $\fr{s}$ and $\fr{p}$ are equivalent it follows that 

\begin{equation}\label{2dimension}\dim(\fr{g})=2\dim(\fr{k}).\end{equation}

\noindent In view of the list of strongly isotropy irreducible spaces $G/K$ including the compact irreducible symmetric spaces in \cite{Bes} p. 201-203, it follows that no such space $G/K$ satisfies condition \eqref{2dimension}. We conclude that $K$ is weakly regular. \end{proof}

Reviewing the proof of Theorem \ref{split2}, we observe that the assumption that $K$ is a regular subgroup of $G$ is only necessary for the implementation of Lemma \ref{key}.  Therefore, Theorem \ref{split2} also holds if $K$ is a weakly regular subgroup of $G$.  More specifically, we have the following generalization of Corollary \ref{ulcorol} for weakly regular subgroups.

\begin{corol}\label{ulcorol1}
Let $G$ be a compact Lie group with Lie algebra $\fr{g}$ and let $K$ be a connected weakly regular subgroup of $G$ with Lie algebra $\fr{k}$.  Moreover, consider the orthogonal decomposition $\fr{g}=\fr{n}_{\fr{g}}(\fr{k})\oplus \fr{p}$ with respect to an $\op{Ad}$-invariant inner product $Q$ on $\fr{g}$.  Then the metric endomorphism $\Lambda=\Lambda_Q\in\op{End}(\fr{g})$ of any $G\times K$-g.o. metric $g$ on $G$ admits the block-diagonal form

\begin{equation}\label{metricendomorphismnaturally2}\Lambda=\begin{pmatrix} 
 \left.\Lambda\right|_{\fr{z}(\fr{n}_{\fr{g}}(\fr{k}))} & 0 & \cdots &0 &0\\
 0& \left.\lambda_1\op{Id}\right|_{\fr{n}_1} &\cdots &0 &0\\
  \vdots & \cdots & \ddots &\vdots &\vdots\\
  0&\cdots &\cdots &\left.\lambda_s\op{Id}\right|_{\fr{n}_s} &0\\
  0&\cdots &\cdots &0 &\left.\Lambda\right|_{\fr{p}}
  \end{pmatrix},
\end{equation}

\noindent where the block $\left.\Lambda\right|_{\fr{p}}$ defines a $G$-g.o. metric on the homogeneous space $G/N^0(K)$,  $\fr{z}(\fr{n}_{\fr{g}}(\fr{k}))$ is the center of $\fr{n}_{\fr{g}}(\fr{k})$ and $\fr{n}_j$, $j=1,\dots, s$, are the simple ideals of $\fr{n}_{\fr{g}}(\fr{k})$.\end{corol} 

Combined with Theorem \ref{naturallyDaZi}, Corollary \ref{ulcorol1} yields the following.

\begin{theorem}\label{ulcorol2}
Let $G$ be a compact simple Lie group with Lie algebra $\fr{g}$ and let $K$ be a connected weakly regular subgroup of $G$ with Lie algebra $\fr{k}$. Assume that any $G$-g.o. metric on the homogeneous space $G/N^0(K)$ is standard (i.e. a scalar multiple of the Killing form).  Then any $G\times K$-g.o. metric on $G$ is naturally reductive.\end{theorem}

\begin{proof} If any $G$-g.o. metric on the homogeneous space $G/N^0(K)$ is standard then $\left.\Lambda\right|_{\fr{p}}=\lambda\left.\op{Id}\right|_{\fr{p}}$ in relation \eqref{metricendomorphismnaturally2}.  The result then follows from Theorem \ref{naturallyDaZi}.\end{proof}

\begin{prop}\label{isotropcorol} Let $(G,g)$ be a connected compact Riemannian simple Lie group and assume that the metric $g$ is $G\times K$-invariant, where $K$ is a subgroup of $G$ such that $G/K$ is strongly isotropy irreducible.   Then $(G,g)$ is a g.o. manifold if and only if it is a naturally reductive manifold.  \end{prop}  

\begin{proof} Clearly, if $(G,g)$ is naturally reductive then it is a g.o. manifold.  Conversely, assume that $(G,g)$ is a g.o. manifold.  We will prove that it is naturally reductive.  Since the metric $g$ is $G\times K$-invariant, Lemma \ref{aux} implies that the isometry group of $(G,g)$ is locally isomorphic to $G\times K^{\prime}$, where $K^{\prime}$ is a connected closed subgroup of $G$ such that $K^0\subseteq K^{\prime}$.  Let $\fr{k}$ and $\fr{k}^{\prime}$ be the Lie algebras of the groups $K$ and $K^{\prime}$ respectively. We have $\fr{k}\subseteq \fr{k}^{\prime}$.  On the other hand, the proper inclusion does not hold, for otherwise we would have reductive decompositions $\fr{g}=\underbrace{\fr{k}\oplus \fr{m}_1}_{\fr{k}^{\prime}}\oplus \fr{m}_2$, and thus $\fr{k}$ would not act irreducibly on the tangent space $\fr{m}_1\oplus \fr{m}_2$ of $G/K$.  Therefore, $\fr{k}=\fr{k}^{\prime}$, $K^0=K^{\prime}$ and thus $(G,g)$ is a $G\times K^0$-g.o. manifold.    

Now by Proposition \ref{corol000}, $\fr{k}$ is self-normalizing in $\fr{g}$.  Since $G/K$ is strongly isotropy irreducible, any $G$-g.o. metric on $G/K$ is standard and thus any $G$-g.o. metric on $G/K^0$ is standard.  Since any $G$-g.o. metric on $G/K^0$ induces a $G$-g.o. metric on $G/N^0(K^0)$ (Corollary \ref{nikoniko}), it follows that any $G$-g.o. metric on $G/N^0(K^0)$ is standard. Theorem \ref{ulcorol2} then implies that the $G\times K^0$-g.o. manifold $(G,g)$ is naturally reductive.\end{proof}

We close the section and the paper with the proof of Theorem \ref{extend}.\\

\noindent \emph{Proof of Theorem \ref{extend}}.  In \cite{ArSaSt} and \cite{CheCheDe3}, the authors obtain non-naturally reductive metrics $g$ on $G=SO(n)$, which are $G\times K$-invariant with $K=SO(k_1)\times SO(k_2)\times SO(k_3)$, $k_1+k_2+k_3=n$.  For the non-naturally reductive metrics found in the aforementioned works, we have $k_i\geq 2$.  Therefore, $n>4$ and thus $SO(n)$ is simple.  The group $K$ is embedded as diagonal block-matrices in $G$.  By Lemma \ref{aux}, the connected isometry group of $(G,g)$ is locally isomorphic to $G\times K^{\prime}$, where $K^{\prime}$ is a closed connected subgroup of $G$ with $K\subseteq K^{\prime}$.  Assume initially that $K=K^{\prime}$.  Along with the facts that any g.o. metric on $G/K$ is standard (\cite{ArSoSt1}) and $K$ is a weakly regular subgroup of $G$ whose Lie algebra is self-normalizing in $\fr{g}$ (Example \ref{examplee}), it follows that any $\widetilde{G}$-g.o. metric on the universal cover $\widetilde{G}/\widetilde{N^0(K)}$ is standard.  By, Proposition \ref{UnivCover}, any $G$-g.o. metric on $G/N^0(K)$ is standard. Theorem \ref{ulcorol2} then implies that $(G,g)$ is not a $G\times K$-g.o. manifold. Since $G\times K$ is locally the isometry group of $(G,g)$, it follows that $(G,g)$ is not a g.o. manifold.
  
   Assume now that $K\subsetneq K^{\prime}$. The metrics in \cite{ArSaSt} and \cite{CheCheDe3} have the form 

\begin{equation}\label{genmet}\langle \ , \ \ \rangle=x_1\left.Q\right|_{\fr{so}(k_1)\times \fr{so}(k_1)}+x_2\left.Q\right|_{\fr{so}(k_2)\times \fr{so}(k_2)}+x_3\left.Q\right|_{\fr{so}(k_3)\times \fr{so}(k_3)}+x_4\left.Q\right|_{\fr{m}_{12}\times \fr{m}_{12}}+x_5\left.Q\right|_{\fr{m}_{13}\times \fr{m}_{13}}+x_6\left.Q\right|_{\fr{m}_{23}\times \fr{m}_{23}}, \end{equation}

\noindent where $Q$ is the negative of the Killing form of $\fr{so}(n)$ and the submodules $\fr{m}_{ij}$ are given in Example \ref{examplee}.  We also have the relations (see \cite{ArSoSt1})

\begin{equation*}\label{temp1}  
[\fr{so}(k_i),\fr{m}_{lm}]=\left\{ 
\begin{array}{lll}
\fr{m}_{lm},  \quad \makebox{if $i=l$ or $i=m$} \ \\

\{0\} ,   \quad \makebox{otherwise}
\end{array}
\right., \ \ \ 0\leq i\leq 3, \ \ 0\leq l<m \leq 3, 
\end{equation*}
and 
\begin{equation*}\label{SubmoduleBrackets}
[\fr{m}_{ij},\fr{m}_{jl}]=\fr{m}_{il} \ \ \makebox{for all} \ \ 0\leq i<j<l\leq 3.
\end{equation*}

Taking into account the above relations, the form \eqref{genmet} of the metrics as well as the fact that the Lie algebra $\fr{k}^{\prime}$ of $K^{\prime}$ is the maximal subalgebra of $\fr{g}$ such that any operator $\op{ad}_X$, $X\in \fr{k}^{\prime}$, is skew-symmetric with respect to $\langle \ ,\ \rangle$, we deduce that the only possibilities for $K^{\prime}$ (so that the metrics are $\op{Ad}_{K^{\prime}}$-invariant) are $K^{\prime}=SO(n)$ (if all $x_i$ are equal), $K^{\prime}=SO(k_1+k_2)\times SO(k_3)$ (if $x_1=x_2=x_4$, $x_5=x_6$), $K^{\prime}=SO(k_1+k_3)\times SO(k_2)$ (if $x_1=x_3=x_5$, $x_4=x_6$) or $K^{\prime}=SO(k_2+k_3)\times SO(k_1)$ (if $x_2=x_3=x_6$, $x_1=x_5$).  The first case for $K^{\prime}$ is not possible, for otherwise the metric would be bi-invariant and hence naturally reductive.  In the remaining cases, $G/K^{\prime}$ is strongly isotropy irreducible (\cite{Wo}) and the fact that $(G,g)$ is not geodesic orbit follows from Proposition \ref{isotropcorol}. 

We have concluded above that the metrics \eqref{genmet} are g.o. if and only if they are naturally reductive.  Assume now that $K=SO(k_1)\times \cdots \times SO(k_s)$, $k_1+\cdots +k_s=n$.  Using similar arguments, we can show that all $G\times K$-invariant non-naturally reductive metrics on $SO(n)$ that generalize the form \eqref{genmet} are not geodesic orbit. More specifically, in \cite{ZaBo} the authors obtain non-naturally reductive metrics $g$ on $G=SO(n)$, $n\geq 12$, which are $G\times K$-invariant with $K=SO(n-9)\times 3 SO(3)$.  The aforementioned metrics have the form 

 \begin{equation}\label{genmet1}\langle \ , \ \rangle=\sum_{i=1}^4 x_i\left.Q\right|_{\fr{so}(k_i)\times \fr{so}(k_i)}+\sum_{1<l<m<4}y_{lm}\left.Q\right|_{\fr{m}_{lm}\times \fr{m}_{lm}}, \end{equation}

\noindent where $\fr{m}_{lm}$ are the pairwise inequivalent and irreducible $\op{Ad}_K$-submodules on the tangent space $T_o(G/K)$. The metrics \eqref{genmet1} generalize the metrics \eqref{genmet}. Using the same arguments as for the metrics \eqref{genmet}, as well as the fact that the metrics \eqref{genmet} are g.o. if and only if they are naturally reductive, we conclude that $(SO(n),\langle \ ,\ \rangle)$ is a g.o. manifold if and only if it is a naturally reductive manifold. Therefore, the non-naturally reductive Einstein metrics of the form \eqref{genmet1} are not geodesic orbit. \\

Finally, we turn our attention to the non-naturally reductive Einstein metrics induced from irreducible triples $(G,K,H)$ of Table 3 in \cite{YanDen}.  Here $H\subsetneq K\subsetneq G$, where $G$ is a compact connected simple Lie group and the isotropy representation of $G/H$ splits into exactly two irreducible submodules.  Those metrics on $G$ are $G\times H$-invariant and are induced from inner products of the form

\begin{equation}\label{standardtriples1}\langle \ ,\ \rangle=\left.Q\right|_{\fr{h}\times\fr{h}}+x \left.Q\right|_{\fr{u}\times\fr{u}}+y\left.Q\right|_{\fr{v}\times\fr{v}},\end{equation}

\noindent where $Q$ is the negative of the Killing form of $\fr{g}$, $\fr{v}$ is the $Q$-orthogonal complement of $\fr{k}$ in $\fr{g}$ and $\fr{u}$ is the $Q$-orthogonal complement of $\fr{h}$ in $\fr{k}$ (here $\fr{k}$, $\fr{h}$ denote the Lie algebras of $K$, $H$ respectively).  As in \cite{YanDen}, we consider the respective Lie algebra triples $(\fr{g},\fr{k},\fr{h})$.

 According to Table 3 in \cite{YanDen}, the only triples that admit non-naturally reductive Einstein metrics of the form \eqref{standardtriples1} are $T_1=(\fr{so}(8),\fr{so}(7),\fr{g}_2)$, $T_2=(\fr{su}(7),\fr{so}(7),\fr{g}_2)$, $T_3=(\fr{su}(8),\fr{so}(8),\fr{so}(7))$, $T_4=(\fr{su}(32),\fr{sp}(16),\fr{so}(12))$, $T_5=(\fr{su}(56),\fr{sp}(28),\fr{e}_7)$ and \\
 $T_6=(\fr{e}_8,\fr{so}(16),\fr{so}(15))$.  According to the classification of the g.o. spaces with two isotropy submodules in \cite{CheNi}, for the triple $T_1$ there exist non-standard g.o. metrics on the corresponding space $G/H$ and hence the assumptions of Theorem \ref{ulcorol2} are not satisfied.  

For the remaining triples $T_2-T_6$, let $G\times K^{\prime}$ be the connected isometry group of $(G, \langle \ ,\ \rangle)$ (up to local isomorphism).  Assuming that $H$ is connected and given that $G/H$ has exactly two isotropy submodules, we deduce that $K^{\prime}$ is one of the groups $G$, $K$ or $H$.  If $K^{\prime}=G$ then the metric $\langle \ ,\ \rangle$ is bi-invariant and hence naturally reductive, which is a contradiction.  If $K^{\prime}=K$ then the metric $\langle \ ,\ \rangle$ is $G\times K$-invariant where $G/K$ is a strongly isotropy irreducible space.  The fact that $(G,\langle \ ,\ \rangle)$ is not g.o. follows from Proposition \ref{isotropcorol}.  

Finally, assume that $K^{\prime}=H$.  For all triples $T_2-T_6$, any $G$-invariant g.o. metric on the corresponding space $G/H$ is standard (\cite{CheNi}). Since any $G$-g.o. metric on $G/H$ induces a $G$-g.o. metric on $G/N^0(H)$ (Corollary \ref{nikoniko}), it follows that any $G$-g.o. metric on $G/N^0(H)$ is standard. Therefore, if we prove that $H$ is a weakly regular subgroup of $G$, Theorem \ref{ulcorol2} will imply that $(G, \langle \ ,\ \rangle)$ is not a g.o. manifold and thus the proof will be concluded.  

Since the isotropy representation of $G/H$ splits into exactly two irreducible submodules, $\fr{k}$ is a maximal subalgebra of $\fr{g}$ containing $\fr{h}$.  However, we observe that $\fr{k}$ is simple in the triples $T_2-T_6$ and thus $\fr{h}$ cannot be a normal subgroup of $\fr{k}$. We conclude that $\fr{n}_{\fr{g}}(\fr{h})=\fr{h}$.  Let $\fr{p}$ be the $Q$-orthogonal complement of $\fr{h}=\fr{n}_{\fr{g}}(\fr{h})$ in $\fr{g}$, which coincides with $T_o(G/H)$. Again, since the isotropy representation of $G/H$ splits into exactly two irreducible submodules, the non-zero $\op{ad}_{\fr{h}}$-irreducible submodules of $\fr{p}$ are precisely (up to possible equivariant isomorphisms) the $Q$-orthogonal complements $\fr{v}$, of $\fr{k}$ in $\fr{g}$, and $\fr{u}$, of $\fr{h}$ in $\fr{k}$. On the other hand, the simple Lie algebra $\fr{h}$ is the only non-zero irreducible $\op{ad}_{\fr{h}}$-submodule of $\fr{h}$.  For all triples $T_2-T_6$, it is not hard to see that the dimensions $\dim(\fr{v})=\dim(\fr{g})-\dim(\fr{k})$ and $\dim(\fr{u})=\dim(\fr{k})-\dim(\fr{h})$ are different from the dimension of $\fr{h}$.  Therefore, $\fr{h}$ cannot be $\op{ad}_{\fr{h}}$-equivalent to any of the irreducible submodules $\fr{v}$ and $\fr{u}$ of $\fr{p}=\fr{u}\oplus \fr{v}$. Since $\fr{n}_{\fr{g}}(\fr{h})=\fr{h}$, it follows from Definition \ref{weaklyregulardef} that $H$ is a weakly regular subgroup of $G$. \qed

\end{document}